\documentclass[reqno]{amsart}
\usepackage{relsize}
\usepackage{amsmath}
\usepackage{amsthm}
\usepackage{amssymb}
\usepackage{hyperref}
\usepackage{comment}
\usepackage{xcolor}
\usepackage{todonotes}

\DeclareMathOperator{\restr}{\!\upharpoonright\!}

\newcommand{\res}{\mathrm{|}}

\newcommand{\ZFC}{{\rm ZFC}}

\renewcommand{\emptyset}{\varnothing}

\newcommand{\R}{{\mathcal R}}
\renewcommand{\S}{{\mathcal S}}

\newcommand{\I}{{\mathcal I}}
\renewcommand{\O}{{\mathcal O}}

\newcommand{\Ult}{\mathop{\rm Ult}}

\newcommand{\restrict}{\upharpoonright}
\newcommand{\concat}{\mathbin{{}^\smallfrown}}

\newcommand{\<}{\langle}
\renewcommand{\>}{\rangle}

\newcommand{\st}{\mid}

\newcommand{\dom}{\mathop{\rm dom}}

\newcommand{\ot}{\mathop{\rm ot}\nolimits}

\newcommand{\code}{{\rm code}}

\newcommand{\NS}{{\mathop{\rm NS}}}

\renewcommand{\and}{\mathop{\&}}

\newtheorem{theorem}{Theorem}[section]
\newtheorem{lemma}[theorem]{Lemma}
\newtheorem{corollary}[theorem]{Corollary}
\newtheorem{proposition}[theorem]{Proposition}

\newtheorem{fact}[theorem]{Fact}
\newtheorem{observation}[theorem]{Observation}
\newtheorem{counterexample}[theorem]{Counterexample}

\theoremstyle{definition}
\newtheorem{question}[theorem]{Question}
\newtheorem{remark}[theorem]{Remark}

\newtheorem{definition}[theorem]{Definition}

\thanks{The research of the second author was supported by the Italian PRIN 2017 Grant \emph{Mathematical Logic: models, sets, computability}.}
\date{\today}

\begin{document}

\title{Ideal operators and higher indescribability}

\author[Brent Cody]{Brent Cody}
\address[Brent Cody]{ 
Virginia Commonwealth University,
Department of Mathematics and Applied Mathematics,
1015 Floyd Avenue, PO Box 842014, Richmond, Virginia 23284, United States
} 
\email[B. ~Cody]{bmcody@vcu.edu} 
\urladdr{http://www.people.vcu.edu/~bmcody/}

\author[Peter Holy]{Peter Holy}
\address[Peter Holy]{
University of Udine,
Via delle Scienze 206, 33100 Udine, Italy
}
\email[P. ~Holy]{pholy@math.uni-bonn.de}

\begin{abstract}
  We investigate properties of the ineffability and the Ramsey operator, and a common generalization of those that was introduced by the second author, with respect to higher indescribability, as introduced by the first author. This extends earlier investigations on the ineffability operator by James Baumgartner, and on the Ramsey operator by Qi Feng, by Philip Welch et al. and by the first author.
\end{abstract}



\maketitle





%
%
%


\section{Introduction}\label{section_introduction}

In the set theoretic literature, some of the most popular large cardinals have been equipped with canonical ideals, and sometimes also with certain operators on ideals. Early examples of such \emph{large cardinal operators} are the \emph{ineffability operator} $\mathcal I$ due to Baumgartner in \cite{MR0540770}, and the \emph{Ramsey operator} $\mathcal R$ that was introduced and extensively studied by Feng in \cite{MR1077260}. 

In the present paper, we want to analyze the interplay of certain large cardinal operators, and in particular the operators $\I$ and $\R$, with a notion of higher indescribability that was introduced by the first author in \cite{CodyHigherIndescribability}, and which extends a notion of Bagaria from \cite{MR3894041}. Let us note that Sharpe and Welch also introduced a notion of higher indescribability \cite[Definition 3.21]{MR2817562}, but the relationship between the notion we use in the present paper and that of \cite{MR2817562} is not currently known.

In the remainder of this section, we recall the definitions of the operators $\I$ and $\R$, and the \emph{strongly Ramsey subset operator} $\S$ (the latter has been first introduced in \cite{HolyLCOandEE}, where it is denoted as $\mathbf{T}_{\mathrm{cl}}$) that is associated to the large cardinal notion of strong Ramseyness, as introduced in \cite{MR2830415}. In Section \ref{section:indescribability}, we review infinitary second order formulas, and their associated notions of indescribability, and we show that many of their basic properties can be established by simple arguments that make use of generic ultrapowers. In Sections \ref{section_generalizing_baumgartner} and \ref{section:indescribability_from_homogeneity}, we generalize results of Baumgartner \cite{MR0384553,MR0540770} to the context of higher indescribability. For example, we show that if $\kappa$ is a subtle cardinal then there are many cardinals $\alpha<\kappa$ which are $\Pi^1_\xi$-indescribable for all $\xi<\alpha^+$. Furthermore, if one assumes the ideal $\I^{\gamma+1}([\kappa]^{<\kappa})$ associated to the $\gamma+1$-ineffability of a cardinal $\kappa$ is nontrivial, where $\kappa\leq\gamma<\kappa^+$, then for any fixed bijection $b:\kappa\to\gamma$ the set
\[\{\alpha<\kappa\st \alpha\in\I^{\text{o.t.}(b[\alpha])}(\Pi^1_\xi(\alpha))^+\}\]
is in the filter dual to $\I^{\gamma+1}([\kappa]^{<\kappa})$ (see Corollary \ref{corollary_below_gamma_plus_1_almost_ineffability}). In Section \ref{section:basic}, we provide two basic lemmas on iterations of the ineffability and the Ramsey operator, which can be viewed as generalizations of the following standard facts: whenever a set has stationarily-many stationary initial segments it must be stationary and if a cardinal is weakly compact then the set of smaller cardinals which are not weakly compact is a weakly compact set. In Section \ref{section:expressibility}, we show that the ideals associated to higher indescribability, and the results of applying our ideal operators to these ideals, can be described by certain infinitary second order formulas. In Section \ref{section_generalized operators}, we review a uniform framework for large cardinal operators from \cite{HolyLCOandEE}, that in particular includes the operators $\I$, $\R$ and $\S$.\footnote{For readers who are only interested in the operators $\I$ and $\R$ (and perhaps also $\S$), it should be possible to skip Section \ref{section_generalized operators}, and only look up some of its relevant bits when necessary. In fact, it is only in Section \ref{section_generating} that we will make use of these generalized operators. We will thus provide some further information regarding this towards the beginning of Section \ref{section_generating}.}
In Section \ref{section:preoperators}, we review the notion of pre-operators. In Section \ref{section_generating}, we combine all of the ingredients to verify yet another generalization of results by Baumgartner, which we extend both to the context of higher indescribability and to the above-mentioned framework for large cardinal operators. For example, suppose $\xi<\kappa^+$ and let $\O$ be the ineffability operator $\I$ or the Ramsey operator $\R$. Then $\O^\omega(\Pi^1_\xi(\kappa))=\O^\omega(\Pi^1_{\xi+1}(\kappa))$, but $\O^n(\Pi^1_\xi(\kappa))\subsetneq\O^n(\Pi^1_{\xi+1}(\kappa))$ for all $n<\omega$ with $\kappa\in\O^n(\Pi^1_{\xi+1}(\kappa))^+$ (see Corollary \ref{corollary_collapse} and Corollary \ref{corollary_proper_containment_from_xi_to_xi_plus_one}). In Section \ref{Cody_results}, we comment on some partially problematic results of the first author from \cite{CodyHigherIndescribability}. In particular, let us point the reader to a simple question involving the Ramsey operator and $\Pi^1_1$-indescribability, namely Question \ref{question_simple}, which has so far resisted all attempts at a resolution.

Without further mention, we will require all ideals to be ideals on some regular and uncountable cardinal $\kappa$, and to be supersets of the bounded ideal on $\kappa$. 
For any ideal $I$, $I^+$ denotes the collection of $I$-positive sets, that is, those subsets of $\kappa$ which are not in $I$, while $I^*$ denotes the filter that is dual to $I$, that is, the collection of complements of sets in $I$. We will often introduce ideals by defining the collection of their positive sets when this is more convenient.

The definition of the Ramsey operator that is provided below is not the original definition from \cite{MR1077260}, but a version that is known to be equivalent \cite[Proposition 2.8]{MR4206111}. Recall that for a set of ordinals $A$, an \emph{$A$-list} is a sequence $\langle a_\alpha\mid \alpha\in A\rangle$ such that $a_\alpha\subseteq \alpha$ for any $\alpha\in A$, and that a set $H\subseteq A$ is homogeneous for $\langle a_\alpha\mid \alpha\in A\rangle$ in case $a_\alpha=a_\beta\cap x$ whenever $\alpha<\beta$ are both in $H$.

\begin{definition}
  Let $I$ be an ideal on $\kappa$.
  \begin{itemize}

    \item Given a $\kappa$-list $\vec a$, we define the \emph{local instance of $\I$ at $\vec{a}$}, \[\mathcal I^{\vec a}(I)^+=\{x\subseteq\kappa\mid\exists H\in I^+\ H\subseteq x\textrm{ is homogeneous for }\vec a\},\] and let $\mathcal I(I)^+=\bigcap\{\mathcal I^{\vec a}(I)^+\mid\vec a$ is a $\kappa$-list$\}$.

    \item Given a regressive function $c\colon[\kappa]^{<\omega}\to\kappa$, we define the \emph{local instance of $\R$ at $c$}, \[\mathcal R^c(I)^+=\{x\subseteq\kappa\mid\exists H\in I^+\ H\subseteq x\textrm{ is homogeneous for }c\},\] and let $\mathcal R(I)^+=\bigcap\{\mathcal R^c(I)^+\mid c\colon[x]^{<\omega}\to\kappa\textrm{ regressive}\}$.

  \end{itemize}
\end{definition}

Feng has shown that $\R(I)$ is a normal ideal on $\kappa$ for any ideal $I$ on $\kappa$, and an analogous result for the operator $\I$ is essentially due to Baumgartner (see also our Lemma \ref{lemma:ineffability_normal} below).
Some of the classical large cardinal ideals (see \cite{MR0384553}) are directly generated by these two large cardinal operators.

\begin{fact}\label{fact:basics}\quad
  \begin{enumerate}
    \item If $\kappa$ is weakly ineffable, then $\mathcal I([\kappa]^{<\kappa})$ is the weakly ineffable ideal on $\kappa$.
    \item If $\kappa$ is ineffable, then $\mathcal I(\NS_\kappa)$ is the ineffable ideal on $\kappa$.
    \item If $\kappa$ is Ramsey, then $\mathcal R([\kappa]^{<\kappa})$ is the Ramsey ideal on $\kappa$.
    \item If $\kappa$ is ineffably Ramsey, then $\mathcal R(\NS_\kappa)$ is the ineffably Ramsey ideal on $\kappa$.
  \end{enumerate}
\end{fact}

In order to define the strongly Ramsey subset operator $\S$, let us recall that $M$ is a \emph{$\kappa$-model} if $M\supseteq\kappa+1$ is a transitive model of $\ZFC^-$ of size $\kappa$ with $M^{<\kappa}\subseteq M$. An \emph{$M$-ultrafilter $U$ on $\kappa$} is a filter $U\subseteq P(\kappa)^M$ which measures all subsets of $\kappa$ in $M$. Also recall that an $M$-ultrafilter $U$ on $\kappa$ is \emph{$\kappa$-amenable for $M$} if whenever $\mathcal A\in M$ is a $\kappa$-sized collection of subsets of $\kappa$ in $M$, then $\mathcal A\cap U\in M$.

\begin{definition}
Let $I$ be an ideal on $\kappa$. Given a set $a\subseteq\kappa$, we define the \emph{local instance of $\S$ at $a$} by letting $x\in\S^a(I)^+$ if and only if $x\subseteq\kappa$ and there is a $\kappa$-model $M$ with $a\in M$ and there is a $\kappa$-amenable $M$-normal $M$-ultrafilter $U$ on $\kappa$ such that $U\subseteq I^+$ and $x\in U$. We let 
\[\S(I)^+=\bigcap\{\S^a(I)^+\st a\subseteq\kappa\}.\]
\end{definition}

It is easy to see that $\kappa$ is strongly Ramsey (as defined in \cite{MR2830415}) if and only if $\kappa\in\S([\kappa]^{<\kappa})^+$, and furthermore, when $\kappa\in\S(I)^+$, it follows that $\S(I)$ is a nontrivial normal ideal on $\kappa$. If $\kappa$ is strongly Ramsey, then $\S([\kappa]^{<\kappa})$ is the \emph{strongly Ramsey ideal} on $\kappa$, as introduced in \cite{MR4156888}.

In this paper, we will also investigate properties of \emph{iterated large cardinal operators}. If $\mathcal O$ is a large cardinal operator, and $I$ is an ideal, we define $\mathcal O^\gamma(I)$ inductively, setting $\mathcal O^{\gamma+1}(I)=\mathcal O(\mathcal O^\gamma(I))$, and $\mathcal O^\gamma(I)=\bigcup_{\delta<\gamma}\mathcal O^\delta(I)$ when $\gamma$ is a limit ordinal.

\section{Review of higher indescribability}\label{section:indescribability}

\subsection{On the notion of $\Pi^1_\xi$- and $\Sigma^1_\xi$-formulas}

The following definition differs slightly from that of Bagaria \cite[Definition 4.1]{MR3894041} in that we allow for $\Pi^1_\xi$- and $\Sigma^1_\xi$-formulas to contain parameters of various kinds.

\begin{definition}\label{definition_over}
Suppose $\kappa$ is a regular cardinal. We define the notions of $\Pi^1_\xi$- and $\Sigma^1_\xi$-formula over $V_\kappa$, for all ordinals $\xi$ as follows. 
\begin{enumerate}
\item A formula $\varphi$ is $\Pi^1_0$, or equivalently $\Sigma^1_0$, over $V_\kappa$ if it is a first order formula in the lanugage of set theory, however we allow for free variables and parameters from $V_\kappa$ of two types, namely of first and of second order. 
\item A formula $\varphi$ is $\Pi^1_{\xi+1}$ over $V_\kappa$ if it is of the form $\forall X_{k_1}\cdots\forall X_{k_m}\psi$ where $\psi$ is $\Sigma^1_\xi$ over $V_\kappa$ and $m\in\omega$. Similarly, $\varphi$ is $\Sigma^1_{\xi+1}$ over $V_\kappa$ if it is of the form $\exists X_{k_1}\cdots\exists X_{k_m}\psi$ where $\psi$ is $\Pi^1_\xi$ over $V_\kappa$ and $m\in\omega$.\footnote{We follow the convention that uppercase letters represent second order variables, while lower case letters represent first order variables. Thus, in the above, all quantifiers displayed are understood to be second order quantifiers, i.e., quantifiers over subsets of $V_\kappa$.}
\item When $\xi$ is a limit ordinal, a formula $\varphi$, with finitely many second-order free variables and finitely many second-order parameters, is $\Pi^1_\xi$ over $V_\kappa$ if it is of the form
\[\bigwedge_{\zeta<\xi}\varphi_\zeta\]
where $\varphi_\zeta$ is $\Pi^1_\zeta$ over $V_\kappa$ for all $\zeta<\xi$. Similarly, $\varphi$ is $\Sigma^1_\xi$ if it is of the form 
\[\bigvee_{\zeta<\xi}\varphi_\zeta\]
where $\varphi_\zeta$ is $\Sigma^1_\zeta$ over $V_\kappa$ for all $\zeta<\xi$.
\end{enumerate}
\end{definition}

Before we can introduce the concept of higher indescribability, which is based on these formula classes, we will need to review a number of further preliminaries.

\subsection{Canonical functions}

For a regular uncountable cardinal $\kappa$, the definition of the $\Pi^1_\xi$-indes\-cribability of a set $S\subseteq\kappa$, where $\kappa\leq\xi<\kappa^+$, that was introduced in \cite{CodyHigherIndescribability}, uses a sequence of functions $\<F^\kappa_\xi\st\xi<\kappa^+\>$, referred to as a \emph{sequence of canonical reflection functions at $\kappa$}, which is defined as follows. If $\xi<\kappa$ then we let $F^\kappa_\xi(\alpha)=\xi$ for all $\alpha\in \kappa$. If $\xi\in\kappa^+\setminus\kappa$, fix a bijection $b_{\kappa,\xi}:\kappa\to\xi$ and let $F^\kappa_\xi(\alpha)=b_{\kappa,\xi}[\alpha]$ for all $\alpha<\kappa$. Notice that for each $\xi<\kappa^+$, the definition of the $\xi^{th}$ canonical reflection function $F^\kappa_\xi$ is independent, modulo the nonstationary ideal, of which bijection $b_{\kappa,\xi}$ is chosen. That is, if $b^1_{\kappa,\xi},b^2_{\kappa,\xi}:\kappa\to\xi$ are two bijections, then the set $\{\alpha<\kappa\st b^1_{\kappa,\xi}[\alpha]=b^2_{\kappa,\xi}[\alpha]\}$ contains a club subset of $\kappa$.

We obtain a sequence of canonical functions $\<f^\kappa_\xi\st\xi<\kappa^+\>$ at $\kappa$ by letting $f^\kappa_\xi(\alpha)=\ot(F^\kappa_\xi(\alpha))$ be the transitive collapse of $F^\kappa_\xi(\alpha)$, for all $\xi<\kappa^+$ and all $\alpha<\kappa$. For all such $\alpha$ and $\xi$, let 
$\pi^\kappa_{\xi,\alpha}:F^\kappa_\xi(\alpha)\to f^\kappa_\xi(\alpha)$
be the transitive collapsing map of $F^\kappa_\xi(\alpha)$. We will assume a fixed choice of these objects throughout the paper.

Intuitively, $\xi$ is to $\kappa$ as $f^\kappa_\xi(\alpha)$ is to $\alpha$, and one can think of $f^\kappa_\xi(\alpha)$ as being $\alpha$'s version of $\xi$ in the sense that when some property involving $\kappa$ and $\xi$ is reflected down to $\alpha$, the statement will be about $\alpha$ and $f^\kappa_\xi(\alpha)$. Notice that (see \cite[Proposition~2.1]{CodyHigherIndescribability}) if $I$ is a normal ideal on $\kappa$, $G$ is generic for $\mathcal P(\kappa)/I$, and $j\colon V\to\Ult$, with $\Ult=V^\kappa/G$, is the corresponding generic ultrapower embedding, then, for all $\xi<\kappa^+$, the $\xi^{th}$ canonical function $f^\kappa_\xi$ represents the ordinal $\xi$ in $\Ult$, that is, $j(f^\kappa_\xi)(\kappa)=[f^\kappa_\xi]_G=\xi$. Similarly, for all $\xi<\kappa^+$, the $\xi^{th}$ canonical reflection function $F^\kappa_\xi$ represents $j"\xi$ in $\Ult$, that is, $j(F^\kappa_\xi)(\kappa)=[F^\kappa_\xi]_G=j"\xi$. Some background material on generic ultrapowers may be found in \cite{MR2768692}, but we will only need very little.

Throughout the rest of our paper, with respect to a regular and uncountable cardinal $\kappa$, let $G$ denote an arbitrary generic filter for $\mathcal P(\kappa)/\NS_\kappa$ over $V$, and let $j\colon V\to \Ult$ be the corresponding generic ultrapower embedding with critical point~$\kappa$. We may sometimes make the extra assumption that $G$ contains some particular stationary subset of $\kappa$ as an element. Note that $\Ult$ may not be well-founded, but it is so up to $\kappa^+$ (as calculated in $V$), and also that $H(\kappa^+)\subseteq\Ult$ and that $H(\kappa)=H(\kappa)^{\Ult}$ in case $\kappa$ is inaccessible.

With respect to the objects that we have fixed at the beginning of this section, at the level of $j(\kappa)$ in $\Ult$, we will always be using the sequence of bijections \[\langle b_{j(\kappa),\xi}\mid j(\kappa)\le\xi<j(\kappa)^+\rangle=j(\langle b_{\kappa,\xi}\mid\kappa\le\xi<\kappa^+\rangle)\] to define the sequences of canonical (reflection) functions that we use.

The following proposition is an easy folklore observation, and will be used to show that the provability of certain statements about generic ultrapowers induces (ground model) statements about canonical (reflection) functions to hold on a club.

\begin{proposition}\label{framework}
Suppose $\kappa$ is a regular uncountable cardinal, $S\subseteq\kappa$ and whenever $G$ is generic for $P(\kappa)/\NS_\kappa$ it follows that $\kappa\in j(S)$ where $j:V\to \Ult$ is the corresponding generic ultrapower. Then $S$ contains a club subset of $\kappa$ in $V$.
\end{proposition}
\begin{proof}
For the sake of contradiction, suppose $S$ does not contain a club subset of $\kappa$ in $V$. Then $T=\kappa\setminus S$ is stationary and we may let $G$ be generic for $P(\kappa)/\NS_\kappa$ with $T\in G$. Then $\kappa\in j(T)$, but this contradicts our assumption that $\kappa\in j(S)$.
\end{proof}


We will need the following lemma later on, which was also presented in \cite[Lemma 2.7 and Lemma 2.8]{CodyHigherIndescribability}, together with easy elementary proofs. For the sake of completeness, and since we will often make use of similar more difficult arguments later on, we would like to provide an even easier proof that makes use of generic ultrapower representations.\footnote{It is fairly straightforward to find generic ultrapower proofs for many further results on canonical (reflection) functions, for example for all the results that are provided in \cite[Section 2]{CodyHigherIndescribability}. We will however not need any such further results in this paper.}

\begin{lemma}\label{lemma_can}
Suppose $\kappa$ is a regular cardinal. For all $\xi<\kappa^+$ the following hold.
\begin{enumerate}
\item If $\xi$ is a limit ordinal, then the set
\[D_0=\{\alpha<\kappa\st\text{$f^\kappa_\xi(\alpha)$ is a limit ordinal}\}\]
is a club subset of $\kappa$.
\item The set
\[D_1=\{\alpha<\kappa\st f^\kappa_{\xi+1}(\alpha)=f^\kappa_\xi(\alpha)+1\}\]
contains a club subset of $\kappa$.
\end{enumerate}
\end{lemma}
\begin{proof}
  \begin{enumerate}
    \item It is easy to see that $D_0$ is closed below $\kappa$. Let $j:V\to\Ult$ be any generic ultrapower obtained by forcing with $P(\kappa)/\NS_\kappa$. Then $j(f^\kappa_\xi)(\kappa)=\xi$ is a limit ordinal and hence $\kappa\in j(D_0)$. By Proposition \ref{framework}, $D_0$ is unbounded in $\kappa$.
    \item Using that $j(f^\kappa_{\xi+1})(\kappa)=\xi+1$, it follows that $\kappa\in j(D_1)$, and the result thus follows by Proposition \ref{framework}.
  \end{enumerate}
\end{proof}

\subsection{Restrictions of formulas}

Let $\kappa$ be a regular and uncountable cardinal throughout. When defining the $\Pi^1_\xi$-indescribability of sets $S\subseteq\kappa$ where $\kappa\leq\xi<\kappa^+$, one cannot simply demand that every $\Pi^1_\xi$-sentence which is true in $V_\kappa$ must be true in $V_\alpha$ for some $\alpha\in S$ because, for example, there are $\Pi^1_\kappa$ sentences with no first or second-order parameters that are true in $V_\kappa$ but which are false in $V_\alpha$ for all $\alpha<\kappa$ (see \cite[Section 1]{CodyHigherIndescribability}). However, one can demand that whenever a $\Pi^1_\xi$-sentence $\varphi$ holds in $V_\kappa$, there must be some $\alpha\in S$ such that a canonically defined restriction $\varphi\res^\kappa_\alpha$ is true in $V_\alpha$. Although we summarize the required background here, one may consult \cite[Sections 3 and 4]{CodyHigherIndescribability} for more information on such canonically defined restrictions of formulas.


In the following, when we talk about either $\Pi^1_\xi$-formulas or $\Sigma^1_\xi$-formulas over $V_\kappa$, for some $\xi<\kappa^+$, we mean formulas which are of that \emph{exact} complexity, and not any simpler one, and we say that these formulas \emph{are of complexity $\xi$}. We also treat such formulas as set theoretic objects, and thus we tacitly assume some reasonable and natural coding of these formulas, and interchangeably use these formulas on the meta level as well as the object level. In particular, we assume that any $\Pi^1_\xi$-formula or $\Sigma^1_\xi$-formula over $V_\kappa$ is coded as an element of $H(\kappa^+)$.

\begin{definition}
By induction on $\xi<\kappa^+$, we define $\varphi\res^\kappa_\alpha$ for all $\Pi^1_\xi$ formulas $\varphi$ over $V_\kappa$ and all regular $\alpha<\kappa$ as follows.\footnote{This is essentially the same definition as in \cite{CodyHigherIndescribability}, however we are being somewhat more careful with respect to the set theoretic representation of formulas here.} First assume that $\xi<\kappa$. If \[\varphi=\varphi(X_1,\ldots,X_m,A_1,\ldots,A_n),\] with free second order variables $X_1,\ldots,X_m$ and second order parameters $A_1,\ldots,A_n$, such that $\alpha>\xi$, and all first order parameters of $\varphi$ are elements of $V_\alpha$, then we define
\[\varphi\res^\kappa_\alpha=\varphi(X_1,\ldots,X_m,A_1\cap V_\alpha,\ldots,A_n\cap V_\alpha),\] and we leave $\varphi\res^\kappa_\alpha$ undefined otherwise.

If $\xi=\zeta+1$ is a successor ordinal and $\varphi=\forall X_{k_1}\ldots\forall X_{k_m}\psi$ is $\Pi^1_{\zeta+1}$ over $V_\kappa$, then we define 
\[\varphi\res^\kappa_\alpha=\forall X_{k_1}\ldots\forall X_{k_m}(\psi\res^\kappa_\alpha)\] in case $\psi\res^\kappa_\alpha$ is defined, and leave $\varphi\res^\kappa_\alpha$ undefined otherwise.
We define $\varphi\res^\kappa_\alpha$ analogously when $\varphi$ is $\Sigma^1_{\zeta+1}$.

If $\xi\in\kappa^+\setminus\kappa$ is a limit ordinal, and 
\[\varphi=\bigwedge_{\zeta<\xi}\psi_\zeta\]
is $\Pi^1_\xi$ over $V_\kappa$, then we define
\[\varphi\res^\kappa_\alpha=\bigwedge_{\zeta\in f^\kappa_\xi(\alpha)}\psi_{(\pi^\kappa_{\xi,\alpha})^{-1}(\zeta)}\res^\kappa_\alpha\] in case 
$\psi_{(\pi^\kappa_{\xi,\alpha})^{-1}(\zeta)}\res^\kappa_\alpha$ is a $\Pi^1_\zeta$-formula over $V_\alpha$ for every $\zeta\in f^\kappa_\xi(\alpha)$. We leave $\varphi\res^\kappa_\alpha$ undefined otherwise.
We define $\varphi\res^\kappa_\alpha$ similarly when $\xi\in\kappa^+\setminus\kappa$ is a limit ordinal and $\varphi$ is $\Sigma^1_\xi$.
\end{definition}

Note that by a simple induction on formula complexity, we obtain the following.

\begin{observation}\label{observation:definedcomplexity}
  If $\varphi$ is a $\Pi^1_\xi$- or $\Sigma^1_\xi$-formula over $V_\kappa$, and $\alpha<\kappa$ is regular, then whenever $\varphi\res^\kappa_\alpha$ is defined, it is a $\Pi^1_{f^\kappa_\xi(\alpha)}$- or $\Sigma^1_{f^\kappa_\xi(\alpha)}$-formula over $V_\alpha$ respectively.
\end{observation}

\begin{remark}\label{remark_coding} We will need the following properties of our coding of formulas. We will leave it to our readers to check that any reasonable coding of formulas has these properties. Assume that $\varphi$ is either a $\Pi^1_\xi$- or $\Sigma^1_\xi$-formula over $V_\kappa$ for some $\xi<\kappa^+$.

\begin{enumerate}
  \item If $\xi<\kappa$, and $A_1,\ldots,A_n$ are all second order parameters appearing in $\varphi$, then \[j(\varphi(A_1,\ldots,A_n))=\varphi(j(A_1),\ldots,j(A_n)).\]
  \item $j(\forall X\,\varphi)=\forall X\,j(\varphi)$.
  \item If $\xi\ge\kappa$ is a limit ordinal, and $\varphi$ is either of the form $\varphi=\bigwedge_{\zeta<\xi}\psi_\zeta$, or of the form $\bigvee_{\zeta<\xi}\psi_\zeta$, let $\vec\psi=\langle\psi_\zeta\mid\zeta<\xi\rangle$. Then, \[j(\varphi)=\bigwedge_{\zeta<j(\xi)}j(\vec\psi)_\zeta\quad\textrm{or}\quad j(\varphi)=\bigvee_{\zeta<j(\xi)}j(\vec\psi)_\zeta\] respectively.
\end{enumerate}
\end{remark}

We will need the following.

\begin{observation}\label{observation:piisjinverse}
  Let $\kappa$ be a regular uncountable cardinal, and let $\xi<\kappa^+$. Let $\vec\pi=\langle\pi^\kappa_{\xi,\alpha}\mid\alpha<\kappa\rangle$. Then $j(\vec\pi)_\kappa^{-1}=j\upharpoonright\xi$.
\end{observation}
\begin{proof}
  Since each $\pi^\kappa_{\xi,\alpha}$ is the transitive collapse of $F^\kappa_\xi(\alpha)$, it follows by elementarity that $j(\vec\pi)_\kappa$ is the transitive collapse of $j(F^\kappa_\xi)(\kappa)=j"\xi$. Hence, the image of $j(\vec\pi)_\kappa$ is $\xi$, and its inverse is thus clearly identical to $j\upharpoonright\xi$.
\end{proof}

The proof of the next lemma is essentially the same as the first part of the proof of \cite[Proposition 3.8]{CodyHigherIndescribability}. Regarding the assumption of the next lemma, and also of some later results, note that $\kappa$ is regular in $\Ult$ if and only if $G$ contains the set of regular cardinals below $\kappa$. This is of course only possible if that latter set is a stationary subset of $\kappa$, i.e., if $\kappa$ is weakly Mahlo.

\begin{lemma}\label{lemma_restrictionvsj}
  If $\varphi$ is either a $\Pi^1_\xi$- or $\Sigma^1_\xi$-formula over $V_\kappa$ for some $\xi<\kappa^+$ and $\kappa$ is regular in $\Ult$, then in $\Ult$, \[j(\varphi)\res^{j(\kappa)}_\kappa=\varphi.\]
\end{lemma}
\begin{proof}
  Regularity of $\kappa$ in $\Ult$ is needed so that $j(\varphi)\res^{j(\kappa)}_\kappa$ could possibly be defined in $\Ult$.
  The proof proceeds by induction on $\xi<\kappa^+$. The case when $\xi<\kappa$ is easy, for then by Remark \ref{remark_coding}(1), \[j(\varphi(A_1,\ldots,A_n))=\varphi(j(A_1),\ldots,j(A_n)),\] and thus, $j(\varphi)\res^{j(\kappa)}_\kappa=\varphi$ by the definition of the restriction operation in this case. Successor steps above $\kappa$ are easily treated as well, for by Remark \ref{remark_coding}(2), in this case, \[j(\forall X\psi(X))=\forall X j(\psi(X)).\]
  
  At limit steps $\xi\ge\kappa$, if $\varphi=\bigwedge_{\zeta<\xi}\psi_\zeta$ is a $\Pi^1_\xi$-formula, let $\vec\psi=\langle\psi_\zeta\mid\zeta<\xi\rangle$, and let $\vec\pi=\langle\pi^\kappa_{\xi,\alpha}\mid\alpha<\kappa\rangle$. Then, by Remark \ref{remark_coding}(3), $j(\varphi)=\bigwedge_{\zeta<j(\xi)}j(\vec\psi)_\zeta$, and therefore \[j(\varphi)\res^{j(\kappa)}_\kappa=\bigwedge_{\zeta\in f^{j(\kappa)}_{j(\xi)}(\kappa)}j(\vec\psi)_{j(\vec\pi)_\kappa^{-1}(\zeta)}\res^{j(\kappa)}_\kappa=\bigwedge_{\zeta\in\xi}(j(\psi_\zeta))\res^{j(\kappa)}_\kappa=\varphi,\]
using that $f^{j(\kappa)}_{j(\xi)}(\kappa)=j(f^\kappa_\xi)(\kappa)=\xi$ by our choice of canonical functions at the level of $j(\kappa)$ in $\Ult$, and by Observation \ref{observation:piisjinverse}. 
  
  The case when $\varphi$ is a $\Sigma^1_\xi$-formula is treated in exactly the same way.
\end{proof}

A neat feature, which could also be seen as a possible motivation for our restriction operation, is now the following.

\begin{lemma}\label{lemma_formularepresentation}
  Assume that $\kappa$ is regular in $\Ult$, that $\varphi$ is either a $\Pi^1_\xi$- or $\Sigma^1_\xi$-formula over $V_\kappa$ for some $\xi<\kappa^+$, and that $\Phi\colon\kappa\to V_\kappa$ is a function with $\Phi(\alpha)=\varphi\res^\kappa_\alpha$ for every regular $\alpha<\kappa$. Then, $\Phi$ represents $\varphi$ in $\Ult$. That is, $j(\Phi)(\kappa)=[\Phi]_U=\varphi$.
\end{lemma}
\begin{proof}
  Note that $j(\Phi)(\kappa)=j(\varphi)\res^{j(\kappa)}_\kappa=\varphi$ by Lemma \ref{lemma_restrictionvsj}.
\end{proof}

The following was essentially shown as \cite[Lemma 3.6]{CodyHigherIndescribability} using an elementary proof, and becomes almost trivial with a generic ultrapower argument.

\begin{lemma}\label{lemma_restriction_is_nice}
  Suppose $\kappa$ is weakly Mahlo. For any $\xi<\kappa^+$, if $\varphi$ is a $\Pi^1_\xi$- or $\Sigma^1_\xi$-formula over $V_\kappa$, then there is a club subset $C$ of $\kappa$ such that for any regular $\alpha\in C$, $\varphi\res^\kappa_\alpha$ is defined, and therefore a $\Pi^1_{f^\kappa_\xi(\alpha)}$- or $\Sigma^1_{f^\kappa_\xi(\alpha)}$-formula over $V_\alpha$ respectively by Observation \ref{observation:definedcomplexity}. 
\end{lemma}
\begin{proof}
  Assume for a contradiction that the conclusion of the lemma fails. This means that there is a stationary set $T$ consisting of regular and uncountable cardinals $\alpha$ such that $\varphi\res^\kappa_\alpha$ is not defined. Assume that $T\in G$. Then, $\kappa\in j(T)$, and therefore $\kappa$ is regular in $\Ult$, however $j(\varphi)\res^{j(\kappa)}_\kappa$ is not defined in $\Ult$. But, by Lemma~\ref{lemma_restrictionvsj}, $j(\varphi)\res^{j(\kappa)}_\kappa=\varphi$, which clearly yields a contradiction.
\end{proof}

We will need the following property of restrictions of formulas, which is established using an argument similar to that of Lemma \ref{lemma_restrictionvsj}.

\begin{lemma}\label{lemma_formularestrictions}
  Suppose that $\kappa$ is regular in $\Ult$. If $\varphi$ is either a $\Pi^1_\xi$- or $\Sigma^1_\xi$-formula over $V_\kappa$ for some $\xi<\kappa^+$, and $\alpha<\kappa$ is regular such that $\varphi\res^\kappa_\alpha$ is defined, then \[j(\varphi)\res^{j(\kappa)}_\alpha=\varphi\res^\kappa_\alpha,\] with the former being calculated in $\Ult$, and the latter being calculated in $V$.
\end{lemma}
\begin{proof}
  By induction on $\xi<\kappa^+$. This is immediate in case $\xi<\kappa$, for then by Remark \ref{remark_coding}(1), $j(\varphi(A_1,\ldots,A_n))=\varphi(j(A_1),\ldots,j(A_n))$, and thus $j(\varphi)\res^{j(\kappa)}_\alpha=\varphi\res^\kappa_\alpha$ by the definition of the restriction operation in this case. It is also immediate for successor steps above $\kappa$, for then by Remark \ref{remark_coding}(2), $j(\forall\vec X\psi)=\forall\vec X j(\psi)$.
  
    At limit steps $\xi\ge\kappa$, if $\varphi=\bigwedge_{\zeta<\xi}\psi_\zeta$ is a $\Pi^1_\xi$-formula, let $\vec\psi=\langle\psi_\zeta\mid\zeta<\xi\rangle$, and let $\vec\pi=\langle\pi^\kappa_{\xi,\alpha}\mid\alpha<\kappa\rangle$. Then, by Remark \ref{remark_coding}(3), $j(\varphi)=\bigwedge_{\zeta<j(\xi)}j(\vec\psi)_\zeta$, and therefore, assuming for now that $j(\varphi)\res^{j(\kappa)}_\alpha$ is defined, \[j(\varphi)\res^{j(\kappa)}_\alpha=\bigwedge_{\zeta\in j(f^\kappa_\xi)(\alpha)}j(\vec\psi)_{j(\vec\pi)_\alpha^{-1}(\zeta)}\res^{j(\kappa)}_\alpha=\bigwedge_{\zeta\in j(f^\kappa_\xi)(\alpha)}j(\psi_{j^{-1}(j(\vec\pi)_\alpha^{-1}(\zeta))})\res^{j(\kappa)}_\alpha,\]
using that $j(\vec\pi)_\alpha^{-1}[j(f^\kappa_\xi)(\alpha)]=j(F^\kappa_\xi)(\alpha)\subseteq j(F^\kappa_\xi)(\kappa)=j"\xi$. 
By our inductive hypothesis, for each $\gamma\in\xi$ and every regular $\alpha<\kappa$, $j(\psi_\gamma)\res^{j(\kappa)}_\alpha=\psi_\gamma\res^\kappa_\alpha$. Thus,
\[j(\varphi)\res^{j(\kappa)}_\alpha=\bigwedge_{\zeta\in j(f^\kappa_\xi)(\alpha)}\psi_{j^{-1}(j(\vec\pi)_\alpha^{-1}(\zeta))}\res^\kappa_\alpha.\]
Now, \[\varphi\res^\kappa_\alpha=\bigwedge_{\zeta\in f^\kappa_\xi(\alpha)}\psi_{(\pi^\kappa_{\xi,\alpha})^{-1}(\zeta)}\res^\kappa_\alpha.\] 
Since $\alpha<\kappa$ we have $j(f^\kappa_\xi)(\alpha)=f^\kappa_\xi(\alpha)$, and furthermore \[(\pi^\kappa_{\xi,\alpha})^{-1}[f^\kappa_\xi(\alpha)]=F^\kappa_\xi(\alpha)=(j^{-1}\circ j(\vec\pi)_\alpha^{-1})[j(f^\kappa_\xi)(\alpha)],\]
showing the above restrictions of $\varphi$ and of $j(\varphi)$ to be equal,\footnote{Being somewhat more careful here, this in fact also uses that the maps $\pi^\kappa_{\xi,\alpha}$, $j$, and $j(\vec\pi)_\alpha$ are order-preserving, so that both of the above conjunctions are taken of the same formulas \emph{in the same order}.} and thus in particular also showing that $j(\varphi)\res^{j(\kappa)}_\alpha$ is defined, as desired.

 The case when $\varphi$ is a $\Sigma^1_\xi$-formula is treated in exactly the same way.
\end{proof}

We can now easily deduce the following, which was originally shown as \cite[Proposition 5.7]{CodyHigherIndescribability}.

\begin{proposition}\label{proposition_double_restriction}
Suppose $\kappa$ is weakly Mahlo, and $\xi<\kappa^+$. For any formula $\varphi$ which is either $\Pi^1_\xi$ or $\Sigma^1_\xi$ over $V_\kappa$, there is a club $D\subseteq\kappa$ such that for all regular uncountable $\alpha\in D$, $\varphi\res^\kappa_\alpha$ is defined, and the set $D_\alpha$ of all ordinals $\beta<\alpha$ such that $(\varphi\res^\kappa_\alpha)\res^\alpha_\beta$ is defined and $(\varphi\res^\kappa_\alpha)\res^\alpha_\beta=\varphi\res^\kappa_\beta$, is in the club filter on $\alpha$. 
\end{proposition}
\begin{proof}
  Assume for a contradiction that the conclusion of the proposition fails. By Lemma \ref{lemma_restriction_is_nice}, this means that there is a stationary set $T$ consisting of regular and uncountable cardinals $\alpha$ such that the set $D_\alpha$ has stationary complement $E_\alpha\subseteq\alpha$. Using Lemma \ref{lemma_restriction_is_nice} once again, we may assume that $(\varphi\res^\kappa_\alpha)\res^\alpha_\beta$ is defined for every $\alpha\in T$ and every $\beta\in E_\alpha$. Let $\vec E$ denote the sequence $\langle E_\alpha\mid\alpha\in T\rangle$. Assume that $T\in G$. Then, $\kappa\in j(T)$, and thus $j(\vec E)_\kappa$ is stationary in $\Ult$. But, \[j(\vec E)_\kappa=\{\beta<\kappa\mid(j(\varphi)\res^{j(\kappa)}_\kappa)\res^\kappa_\beta\ne j(\varphi)\res^{j(\kappa)}_\beta\}.\]
  Note that by Lemma \ref{lemma_formularepresentation}, $j(\varphi)\res^{j(\kappa)}_\kappa=\varphi$. But then, by Lemma \ref{lemma_restriction_is_nice} and Lemma \ref{lemma_formularestrictions}, $j(\vec E)_\kappa$ is nonstationary in $\Ult$, which gives our desired contradiction.
\end{proof}

\subsection{Higher indescribability}

The notion of $\Pi^1_\xi$-indescribability of (subsets of) a cardinal $\kappa$ when $\xi<\kappa$ was introduced by Joan Bagaria in \cite{MR3894041}, and was extended by the first author as follows.

\begin{definition}[{\cite[Definition 3.4]{CodyHigherIndescribability}}]\label{definition_indescribability}
Suppose $\kappa$ is a cardinal and $\xi<\kappa^+$. A set $S\subseteq\kappa$ is \emph{$\Pi^1_\xi$-indescribable} if for every $\Pi^1_\xi$ sentence $\varphi$ over $V_\kappa$, if $V_\kappa\models\varphi$ then there is some $\alpha\in S$ such that $\varphi\res^\kappa_\alpha$ is defined, and $V_\alpha\models\varphi\res^\kappa_\alpha$.
\end{definition}

Note that the value of any particular $\varphi\res^\kappa_\alpha$ depends on our choice of bijections $b_{\kappa,\xi}$, however using Lemma \ref{lemma_restrictionvsj} and Proposition \ref{framework}, it is easy to see that for any two choices of sequences $\langle b_{\kappa,\xi}\mid\xi<\kappa^+\rangle$, the corresponding $\varphi\res^\kappa_\alpha$'s agree on a club, and thus in particular the above notion of higher indescribability is independent of that choice.\footnote{This was also shown using elementary proofs as \cite[Lemma 3.3]{CodyHigherIndescribability} and \cite[Lemma 3.7]{CodyHigherIndescribability}.}

Note also that $S\subseteq\kappa$ is $\Pi^1_0$-indescribable if and only if $S$ is a stationary subset of $\kappa$. We will say that $S\subseteq\kappa$ is $\Pi^1_{-1}$-indescribable in case $S$ is an unbounded subset of $\kappa$. For $\xi\in\{-1\}\cup\kappa^+$, we let $\Pi^1_\xi(\kappa)^+$ be the collection of all $\Pi^1_\xi$-indescribable subsets of $\kappa$. It was shown by the first author in \cite[Theorem 5.5]{CodyHigherIndescribability} that if $\kappa$ is a cardinal, $\xi<\kappa^+$, and $\kappa$ is $\Pi^1_\xi$-indescribable, then $\Pi^1_\xi(\kappa)$ is a nontrivial normal ideal on $\kappa$.

\section{Generalizations of a result of Baumgartner}\label{section_generalizing_baumgartner}

A key result from Baumgartner's \cite{MR0384553} is the following theorem that indicates the strength of subtlety. Recall that $S\subseteq\kappa$ is \emph{subtle} in case whenever $\vec S=\<S_\alpha\st\alpha\in S\>$ is an $S$-list and $C\subseteq\kappa$ is club, then there are $\alpha<\beta$ both in $S\cap C$ such that $S_\alpha=S_\beta\cap\alpha$.

\begin{theorem}[Baumgartner]\cite[Theorem 4.1]{MR0384553}\label{theorem_generalizing_Baumgartner}
  Suppose $S\subseteq\kappa$ is subtle and $\vec S=\<S_\alpha\st\alpha\in S\>$ is an $S$-list. Let
\[A=\{\alpha\in S\st(\exists X\subseteq S\cap\alpha)(\forall \eta<\omega\ X\in\Pi^1_\eta(\alpha)^+)\land(X\cup\{\alpha\}\text{ is hom. for }\vec{S})\}.\]
Then, $S\setminus A$ is not subtle.
\end{theorem}

In this section, we want to provide a strengthening of Baumgartner's theorem with respect to higher indescribability, and then apply this to obtain a related result on iterations of the ineffability operator $\I$.

\subsection{Coding formulas}

When $\kappa$ is inaccessible, we will need a sort of improved coding of $\Pi^1_\xi$- and $\Sigma^1_\xi$-formulas over $V_\kappa$ for $\xi<\kappa^+$, with the property that such formulas over $V_\kappa$ are coded as subsets of $\kappa$.

\begin{definition}
Suppose $\kappa$ is an inaccessible cardinal, and fix a bijection \[b^\kappa\colon V_\kappa\to\kappa.\footnote{The purpose of this bijection will be the coding of parameters of our formulas, and it is only for its existence here that we use the inaccessibility of $\kappa$.}\] For each $\kappa\le\xi<\kappa^+$, fix a well-ordering \[R^\kappa_\xi\subseteq\kappa\times\kappa\] of $\kappa$ such that $\ot(\kappa,R^\kappa_\xi)=\xi$ and let \[b_{\kappa,\xi}:\kappa\to\xi\] be the bijection derived from $R^\kappa_\xi$. We call the above the \emph{coding parameters at $\kappa$}. In what follows, by induction on formula complexity, we define a coding function $\code^\kappa$ such that whenever for some $\xi<\kappa^+$, $\varphi$ is a $\Pi^1_\xi$- or $\Sigma^1_\xi$-formula over $V_\kappa$, $\code^\kappa(\varphi)$ is a subset of $\kappa$. 

If $\xi<\kappa$, we let $\code^\kappa(\varphi)$ be a subset of $\kappa$ coding $\varphi$ in some reasonable way, making use of the bijection $b^\kappa$ to code the parameters of $\varphi$. In particular, we require that its first slice of code, $(\code^\kappa(\varphi))_0=\emptyset$,\footnote{So that we can distinguish this basic case from the later cases.} we use its slices with finite indices to code the second order parameters of $\varphi$, we only use boundedly many nonempty slices to code $\varphi$, and all such slices with infinite index are bounded subsets of $\kappa$.

Let $\varphi$ be a $\Pi^1_\xi$- or $\Sigma^1_\xi$-formula over $V_\kappa$ for some $\kappa\le\xi<\kappa^+$, and assume that we have inductively defined $\code^\kappa(\psi)$ whenever $\psi$ is of lower complexity. We define $\code^\kappa(\varphi)$ to be a subset of $\kappa$ as follows.
\begin{itemize}
\item Suppose $\xi=\zeta+1$ is a successor ordinal. If $\varphi=\forall X_{k_0}\ldots\forall X_{k_m}\psi$ is $\Pi^1_{\zeta+1}$ with $m\in\omega$ and $\psi$ being $\Sigma^1_\zeta$, we define 
\[(\code^\kappa(\varphi))_0=\{k_0,\ldots,k_m,\omega\}\]
and if $\varphi=\exists X_{k_0}\ldots\exists X_{k_m}\psi$ is $\Sigma^1_{\zeta+1}$ with $m\in\omega$ and $\psi$ being $\Pi^1_\zeta$, we define
\[(\code^\kappa(\varphi))_0=\{k_0,\ldots,k_m,\omega+1\}.\]
In either case, we let
\[(\code^\kappa(\varphi))_1=\code^\kappa(\psi),\] and for $1<\nu<\kappa$, we let
\[(\code^\kappa(\varphi))_\nu=\emptyset.\]
\item Suppose $\xi$ is a limit ordinal. If $\varphi=\bigwedge_{\zeta<\xi}\psi_\zeta$ is $\Pi^1_\xi$, we let
\[(\code^\kappa(\varphi))_0=\{0\},\]
and if $\varphi=\bigvee_{\zeta<\xi}\psi_\zeta$ is $\Sigma^1_\xi$, we let
\[(\code^\kappa(\varphi))_0=\{1\}.\]
In either case, we let
\[(\code^\kappa(\varphi))_1=\Gamma[R^\kappa_\xi],\]
where $\Gamma$ denotes the G\"odel pairing function, and for all $\nu<\kappa$ we let
\[(\code^\kappa(\varphi))_{2+\nu}=\code^\kappa(\psi_{b_{\kappa,\xi}(\nu)}).\]
\end{itemize}
\end{definition}

Fix a sequence of coding parameters at $\alpha$ for every inaccessible $\alpha\le\kappa$, such that the sequence $\langle b^\alpha\mid\alpha\le\kappa\rangle$ is $\subseteq$-increasing. Note that, using the $j$-images of our coding parameters to code in $\Ult$, by elementarity, for any relevant formula $\varphi$ over $V_\kappa$, we have \[j(\code^\kappa(\varphi))=\code^{j(\kappa)}(j(\varphi)).\]

\begin{lemma}\label{lemma_code_of_a_restriction}
Suppose $\kappa$ is Mahlo. If $\varphi$ is a $\Pi^1_\xi$- or $\Sigma^1_\xi$-formula over $V_\kappa$ for some $\xi<\kappa^+$, $H\subseteq\kappa$ is stationary and consists only of regular cardinals, and for each $\alpha\in H$, there is a $\Pi^1_{f^\kappa_\xi(\alpha)}$- or $\Sigma^1_{f^\kappa_\xi(\alpha)}$-formula $\varphi^\alpha$ over $V_\alpha$ respectively, such that \[\code^\alpha(\varphi^\alpha)=\code^\kappa(\varphi)\cap\alpha,\] then there is a club $C\subseteq\kappa$ such that for each $\alpha\in H\cap C$, $\varphi\res^\kappa_\alpha$ is defined, and
\[\code^\kappa(\varphi)\cap\alpha=\code^\alpha(\varphi\res^\kappa_\alpha).\]
\end{lemma}
\begin{proof}
  Suppose for a contradiction that the conclusion of the lemma fails. Using Lemma \ref{lemma_restriction_is_nice}, this means that there is a stationary set $T\subseteq H$ such that for every $\alpha\in T$, \[\code^\kappa(\varphi)\cap\alpha\ne\code^\alpha(\varphi\res^\kappa_\alpha).\] Assume $T\in G$. Then, $\kappa\in j(T)$, hence in $\Ult$, $\kappa$ is regular and \[\code^{j(\kappa)}(j(\varphi))\cap\kappa\ne\code^\kappa(j(\varphi)\res^{j(\kappa)}_\kappa).\]
  But by Lemma \ref{lemma_restrictionvsj}, this means that in $\Ult$, 
  \begin{align}\code^{j(\kappa)}(j(\varphi))\cap\kappa\ne\code^\kappa(\varphi).\label{equation_code}\end{align}
  
  Let us show that (\ref{equation_code}) is false, thus yielding our desired contradiction. First, let us consider the case in which $\xi<\kappa$. We have $j(\varphi(A_1,\ldots,A_n))=\varphi(j(A_1),\ldots,j(A_n))$ in this case, and by our choice of bijections we see that in $\Ult$, $b^{j(\kappa)}\supseteq b^\kappa$ and hence first order parameters are coded in the same way. Furthermore, by our choice of \emph{reasonable coding}, we observe that $\code^{j(\kappa)}(j(\varphi))\cap\kappa$ and $\code^\kappa(\varphi)$ have the same slices and hence $\code^{j(\kappa)}(j(\varphi))\cap\kappa= \code^\kappa(\varphi)$.

  Let us inductively look at the cases when $\xi\ge\kappa$.
  The successor ordinal case is immediate, comparing all slices of $\code^{j(\kappa)}(j(\varphi))\cap\kappa$ and of $\code^\kappa(\varphi)$.

  Assume now that $\xi$ is a limit ordinal. We will again be comparing the slices of $\code^{j(\kappa)}(j(\varphi))\cap\kappa$ and of $\code^\kappa(\varphi)$. The slices with index $0$ clearly agree. Our assumption, which we haven't used yet, yields a formula $\psi$ of complexity $\xi$ such that in $\Ult$, \[\code^\kappa(\psi)=\code^{j(\kappa)}(j(\varphi))\cap\kappa.\] The slices with index $1$ thus agree between $\code^\kappa(\psi)$ and $\code^\kappa(\varphi)$, for they are coding the same well-ordering, and hence they also agree for our desired formulas. The remaining slices agree inductively, contradicting (\ref{equation_code}) as desired.
\end{proof}

\subsection{Generalizing Baumgartner's lemma to higher indescribability}\label{subsection_generalizing_Baumgertner}

In this section, we provide the promised strengthening of Theorem \ref{theorem_generalizing_Baumgartner}.

\begin{theorem}\label{theorem_baumgartner}
Suppose $S\subseteq\kappa$ is subtle and $\vec{S}=\<S_\alpha\st\alpha\in S\>$ is an $S$-list. Let
\[A=\{\alpha\in S\st(\exists X\subseteq S\cap\alpha)(\forall \eta<\alpha^+\ X\in\Pi^1_\eta(\alpha)^+)\land(X\cup\{\alpha\}\text{ is hom. for }\vec{S})\}.\]
Then, $S\setminus A$ is not subtle.
\end{theorem}

\begin{proof}
Suppose $S\subseteq\kappa$ is subtle, $\vec{S}$ is an $S$-list, and suppose for a contradiction that $S\setminus A$ is subtle. Using that the set of inaccessible cardinals below $\kappa$ is in the subtle filter on $\kappa$, we may assume that all elements of $S$ are inaccessible. We will also assume our canonical functions at $\kappa$ to be based on the bijections $b_{\kappa,\xi}$ used to define the coding at $\kappa$ in the above.

Suppose $\beta\in S\setminus A$. Let $B_\beta=\{\alpha\in S\cap\beta\st S_\alpha=S_\beta\cap\alpha\}$. Since $B_\beta\cup\{\beta\}$ is homogeneous for $\vec{S}$, it follows that for some limit ordinal $\xi_\beta$ with $\beta\le\xi_\beta<\beta^+$, the set $B_\beta$ is not $\Pi^1_{\xi_\beta}$-indescribable in $\beta$. Let $\varphi^\beta$ thus be a $\Pi^1_{\xi_\beta}$ sentence over $V_\beta$ such that $V_\beta\models\varphi^\beta$, and for all $\alpha\in B_\beta$, we have that $V_\alpha\not\models\varphi^\beta\res^\beta_\alpha$ whenever $\varphi^\beta\res^\beta_\alpha$ is defined. 

For each $\beta\in S\setminus A$, let $E_\beta$ code the pair
\[\langle S_\beta,\code^\beta(\varphi^\beta)\rangle\]
as a subset of $\beta$ in a natural way. This defines an $(S\setminus A)$-list $\vec{E}=\<E_\beta\st\beta\in S\setminus A\>$. By Theorem \ref{theorem_generalizing_Baumgartner}, there is a Mahlo cardinal $\beta\in S\setminus A$ and a stationary set $H\subseteq (S\setminus A)\cap\beta$ such that $H\cup\{\beta\}$ is homogeneous for $\vec{E}$. Let $\varphi=\varphi^\beta$, and let $\xi=\xi_\beta$. 

Recall that for each $\alpha\in S\setminus A$, we have $(\code^\alpha(\varphi^\alpha))_1=\Gamma[R^\alpha_{\xi_\alpha}]$, and thus the homogeneity of $H\cup\{\beta\}$ implies that for all $\alpha\in H$, we have
\[R^\beta_\xi\cap(\alpha\times\alpha)=R^\alpha_{\xi_\alpha},\]
and therefore
\[\xi_\alpha=\ot(\alpha,R^\alpha_{\xi_\alpha})=\ot(\alpha,R^\beta_\xi\cap(\alpha\times\alpha))=\ot(b_{\beta,\xi}[\alpha])=f^\beta_\xi(\alpha).\]

Thus, for each $\alpha\in H$, we have a $\Pi^1_{f^\beta_\xi(\alpha)}$-formula $\varphi^\alpha$ over $V_\alpha$ such that
\[\code^\alpha(\varphi^\alpha)=\code^\beta(\varphi)\cap\alpha.\]
Therefore, by Lemma \ref{lemma_code_of_a_restriction}, there is a club $C\subseteq\beta$ such that $\varphi\res^\beta_\alpha$ is defined, and
\[\code^\beta(\varphi)\cap\alpha=\code^\alpha(\varphi\res^\beta_\alpha)\]
for all $\alpha\in H\cap C\ne\emptyset$. Fix some $\alpha\in H\cap C$. We have
\[\code^\alpha(\varphi^\alpha)=\code^\beta(\varphi)\cap\alpha=\code^\alpha(\varphi\res^\beta_\alpha),\]
and hence $\varphi^\alpha=\varphi\res^\beta_\alpha$. However, since $S_\beta\cap\alpha=S_\alpha$ as well by the homogeneity of $H$, we have $\alpha\in B_\beta$, and hence $V_\alpha\not\models\varphi\res^\beta_\alpha$, contradicting that $V_\alpha\models\varphi^\alpha$.
\end{proof}

The following is now immediate from Theorem \ref{theorem_baumgartner}.

\begin{corollary}\label{corollary_below_a_subtle_cardinal}
Suppose $\kappa$ is subtle. Then, the set
\[\{\alpha<\kappa\st(\forall\eta<\alpha^+)\ \text{$\alpha$ is $\Pi^1_\eta$-indescribable}\}\]
is in the subtle filter on $\kappa$. \hfill{$\Box$}
\end{corollary}

\subsection{Pushing Baumgartner's lemma up the ineffability hierarchy}\label{subsection_pushing_Baumgertner}

In this section, we provide our promised application of Theorem \ref{theorem_baumgartner} on the ineffability hierarchy, showing that iterated applications of the ineffability operator yield strong generalizations of the consequences of subtlety described in Theorem \ref{theorem_baumgartner}. We will also apply this result in order to obtain further results on the ineffability hierarchy later on in our paper. We will first need an easy auxiliary lemma, which is the analogue of \cite[Theorem 2.1]{MR1077260} for the ineffability operator. This result is essentially due to Baumgartner (two particular instances are mentioned as \cite[Theorem 2.3 and Theorem 2.4]{MR0384553}), and can be seen to follow from a combination of several results from \cite{HolyLCOandEE}, and in particular its \cite[Proposition 10.2]{HolyLCOandEE}. We would rather like to provide a self-contained proof, which is a minor adaption of the proof of \cite[Theorem 2.2]{MR0384553}.

\begin{lemma}\label{lemma:ineffability_normal}
  If $I$ is an ideal on $\kappa$, then $\I(I)$ is a normal ideal on $\kappa$.
\end{lemma}
\begin{proof}
  The only nontrivial property is normality. Note first that $\I(I)\supseteq\I([\kappa]^{<\kappa})\supseteq\NS_\kappa$, where the latter is due to an easy argument that may be found within the proof of \cite[Theorem 2.3]{MR0384553}.
  
  Assume now that $A\in\I(I)^+$. Let $C$ be the club set of all ordinals below $\kappa$ that are closed under the G\"odel pairing function $\Gamma$. By the above, it follows that $A\cap C\in\I(I)^+$, and we may thus assume that every element of $A$ is closed under G\"odel pairing.
  
  Assume that $f\colon A\to\kappa$ is regressive, and let $A_\alpha=f^{-1}(\alpha)$ for every $\alpha<\kappa$. Assume for a contradiction that $A_\alpha\in\I(I)$ for every $\alpha<\kappa$. Thus, for every $\alpha<\kappa$, we may fix an $A_\alpha$-list $\vec A_\alpha=\<a^\alpha_\beta\mid\beta\in A_\alpha\>$ which has no homogeneous set in $I^+$. 

Let $\vec A=\<a_\beta\mid\beta\in A\>$ be an $A$-list defined by letting, for every $\beta\in A$, $a_\beta$ code both $\{f(\beta)\}$ and $a^{f(\beta)}_\beta$, using G\"odel pairing. Since $A\in\I(I)^+$, there is $H\in I^+$ that is homogeneous for $\vec A$. It follows that $H$ is homogeneous for $f$, and we let $\alpha$ be the value of $f$ on $H$. It then follows that $H$ is homogeneous for $\vec A_\alpha$. This yields our desired contradiction.
\end{proof}

We need another auxiliary observation, which has already been used by Baumgartner in \cite{MR0384553}, but which we couldn't find a proof of in the set-theoretic literature. For the convenience of our readers, we would therefore like to provide the easy argument.

\begin{observation}
  For any cardinal $\kappa$, the subtle ideal on $\kappa$ is contained in the weakly ineffable ideal on $\kappa$.
\end{observation}
\begin{proof}
  Assume that $A\subseteq\kappa$ is not an element of the weakly ineffable ideal $\I([\kappa]^{<\kappa})$ on $\kappa$, that $\vec a$ is an $A$-list, and let $C$ be a club subset of $\kappa$. By Lemma \ref{lemma:ineffability_normal}, $A\cap C\in\I([\kappa]^{<\kappa})^+$. It follows that we can find $H\subseteq A\cap C$ that is homogeneous for $H$. This shows that $A$ is subtle.
\end{proof}

\begin{theorem}\label{theorem_pushing_Baumgartner}
Suppose $\gamma<\kappa^+$, $S\in \I^{\gamma+1}([\kappa]^{<\kappa})^+$ and $\vec{S}=\<S_\alpha\st\alpha\in S\>$ is an $S$-list. Let $A$ be the set of all ordinals $\alpha\in S$ such that
\[\exists X\subseteq S\cap\alpha\left[(\forall \xi<\alpha^+\ X\in \I^{f^\kappa_\gamma(\alpha)}(\Pi^1_\xi(\alpha))^+) \land (X\cup\{\alpha\}\text{ is hom. for }\vec{S})\right].\]
Then, $S\setminus A\in\I^{\gamma+1}([\kappa]^{<\kappa})$.
\end{theorem}

\begin{proof}
We proceed by induction on $\gamma<\kappa^+$. When $\gamma=0$, the result follows directly from Theorem \ref{theorem_baumgartner}, because the subtle ideal is contained in the weakly ineffable ideal $\I([\kappa]^{<\kappa})$, and $f^\kappa_0(\alpha)=0$ for all $\alpha<\kappa$.

Suppose $\gamma=\delta+1<\kappa^+$ is a successor ordinal, and suppose for a contradiction that $S\setminus A\in\I^{\delta+2}([\kappa]^{<\kappa})^+$. Let $C=\{\alpha<\kappa\st f^\kappa_{\delta+1}(\alpha)=f^\kappa_\delta(\alpha)+1\}$ be the club subset of $\kappa$ obtained from Lemma \ref{lemma_can}. Then, the set
\[E=\{\alpha\in S\setminus A\st \alpha\text{ is inaccessible}\}\cap C\]
is in $\I^{\delta+2}([\kappa]^{<\kappa})^+$. For each $\alpha\in E$, let $B_\alpha=\{\beta\in S\cap\alpha\st S_\beta=S_\alpha\cap\beta\}$. Since $B_\alpha\cup\{\alpha\}$ is homogeneous for $\vec{S}$ and $\alpha\in S\setminus A$, there is an ordinal $\xi_\alpha<\alpha^+$ such that $B_\alpha\in\I^{f^\kappa_\delta(\alpha)+1}(\Pi^1_{\xi_\alpha}(\alpha))$, and hence we may fix a $B_\alpha$-list $\vec{B}^\alpha=\<b^\alpha_\beta\st\beta\in B_\alpha\>$ such that $\vec{B}^\alpha$ has no homogeneous set in $\I^{f^\kappa_\delta(\alpha)}(\Pi^1_{\xi_\alpha}(\alpha))^+$. 


For $\alpha\in E$, let $E_\alpha$ code the triple $\langle S_\alpha,B_\alpha,\vec{B}^\alpha\rangle$ as a subset of $\alpha$ in a natural way. This defines an $E$-list $\vec{E}=\<E_\alpha\st\alpha\in E\>$. Since $E\in \I^{\delta+2}([\kappa]^{<\kappa})^+$, we may fix $H\in \mathcal P(E)\cap\I^{\delta+1}([\kappa]^{<\kappa})^+$ which is homogeneous for $\vec{E}$. It follows that $H$ is homogeneous for $\<S_\alpha\st\alpha\in E\>$, $\<B_\alpha\st\alpha\in E\>$ and $\<\vec{B}^\alpha\st\alpha\in E\>$. We let $D=\bigcup_{\alpha\in H} S_\alpha$, $B=\bigcup_{\alpha\in H}B_\alpha$ and $\vec{B}=\bigcup_{\alpha\in H}\vec{B}^\alpha=\<b_\alpha\st\alpha\in B\>$. Since $B=\{\alpha<\kappa\st S_\alpha=D\cap\alpha\}$, it follows that $H\subseteq B$.

Let $A_0$ be the set of all ordinals $\alpha\in H$ such that
\[\exists X\subseteq H\cap\alpha\left[(\forall\xi<\alpha^+\ X\in \I^{f^\kappa_\delta(\alpha)}(\Pi^1_\xi(\alpha))^+) \land (X\cup\{\alpha\}\text{ is hom. for }\vec{B})\right].\]
By our inductive hypothesis, $H\setminus A_0\in\I^{\delta+1}([\kappa]^{<\kappa})$, and hence $A_0\in \I^{\delta+1}([\kappa]^{<\kappa})^+$. Thus, we may fix an $\alpha\in A_0$. Since $\alpha\in H$, it follows by homogeneity that $\vec{B}\restrict (H\cap\alpha)=\vec{B}^\alpha\restrict H$. But by the definition of $A_0$, and since $\xi_\alpha<\alpha^+$, there is some $X\in \mathcal P(H\cap\alpha)\cap\I^{f^\kappa_\delta(\alpha)}(\Pi^1_{\xi_\alpha}(\alpha))^+$ which is homogeneous for $\vec{B}^\alpha$, which is a contradiction.

Now let us suppose $\gamma<\kappa^+$ is a limit ordinal, and suppose again for a contradiction that $S\setminus A\in\I^{\gamma+1}([\kappa]^{<\kappa})^+$. Since by Lemma \ref{lemma_can}, the set
\[C=\{\alpha<\kappa\st f^\kappa_\gamma(\alpha)\text{ is a limit ordinal and } f^\kappa_\gamma(\alpha)=\bigcup_{\delta\in F^\kappa_\gamma(\alpha)} f^\kappa_\delta(\alpha) \}\]
is in the club filter on $\kappa$, it follows that the set
\[E=\{\alpha\in S\setminus A\st\text{$\alpha$ is inaccessible}\}\cap C\]
is in $\I^{\gamma+1}([\kappa]^{<\kappa})^+$.

For each $\alpha\in E$, let $B_\alpha=\{\beta\in S\cap\alpha\st S_\beta=S_\alpha\cap\beta\}$. Since $B_\alpha\cup\{\alpha\}$ is homogeneous for $\vec{S}$, and $\alpha\in S\setminus A$, there is some $\xi_\alpha<\alpha^+$ such that $B_\alpha\in\I^{f^\kappa_\gamma(\alpha)}(\Pi^1_{\xi_\alpha}(\alpha))$. Since $\alpha\in C$, we have
\[B_\alpha\in\I^{f^\kappa_\gamma(\alpha)}(\Pi^1_{\xi_\alpha}(\alpha))=\bigcup_{\delta\in F^\kappa_\gamma(\alpha)}\I^{f^\kappa_\delta(\alpha)}(\Pi^1_{\xi_\alpha}(\alpha))=\bigcup_{\beta<\alpha}\I^{f^\kappa_{b_{\kappa,\gamma}(\beta)}(\alpha)}(\Pi^1_{\xi_\alpha}(\alpha)).\]
Using that $f^\kappa_\gamma(\alpha)$ is a limit ordinal once again, we may choose an ordinal $g(\alpha)<\alpha$ such that \[B_\alpha\in\I^{f^\kappa_{b_{\kappa,\gamma}(g(\alpha))}(\alpha)+1}(\Pi^1_{\xi_\alpha}(\alpha)).\]
This defines a regressive function $g:E\to\kappa$, and by the normality of $\I^{\gamma+1}([\kappa]^{<\kappa})^+$ that follows from Lemma \ref{lemma:ineffability_normal}, there is an $E^*\in \mathcal P(E)\cap \I^{\gamma+1}([\kappa]^{<\kappa})^+$ and some $\beta_0<\kappa$ such that $g(\alpha)=\beta_0$ for all $\alpha\in E^*$. Let $\nu=b_{\kappa,\gamma}(\beta_0)$ and notice that for all $\alpha\in E^*$, 
\[B_\alpha\in\I^{f^\kappa_{\nu}(\alpha)+1}(\Pi^1_{\xi_\alpha}(\alpha)).\] For each $\alpha\in E^*$, we fix a $B_\alpha$-list $\vec{B}^\alpha=\<b^\alpha_\beta\st\beta\in B_\alpha\>$ such that $\vec{B}^\alpha$ has no homogeneous set in $\I^{f^\kappa_{\nu}(\alpha)}(\Pi^1_{\xi_\alpha}(\alpha))^+$.

Now we define an $E^*$-list by letting $E^*_\alpha$ code the triple $\langle S_\alpha,B_\alpha,\vec{B}^\alpha\rangle$ as a subset of $\alpha$ in a natural way, for all $\alpha\in E^*$. This defines $\vec{E}^*=\<E^*_\alpha\st\alpha\in E^*\>$. Since $E^*\in \I^{\gamma+1}([\kappa]^{<\kappa})^+$, we may fix an $H\in \mathcal P(E^*)\cap\I^\gamma([\kappa]^{<\kappa})^+$ which is homogeneous for $\vec{E}^*$. Then, $H$ is homogeneous for $\<S_\alpha\st\alpha\in S\>$, $\<B_\alpha\st\alpha\in E^*\>$ and $\<\vec{B}^\alpha\st\alpha\in E^*\>$. We let $D=\bigcup_{\alpha\in H}S_\alpha$, $B=\bigcup_{\alpha\in H}B_\alpha$ and $\vec{B}=\bigcup_{\alpha\in H}\vec{B}^\alpha=\<b_\alpha\st\alpha\in B\>$. Since $B=\{\alpha<\kappa\st S_\alpha=D\cap\alpha\}$, it follows that $H\subseteq B$.

Now since $\nu<\gamma$, we have $H\in \I^\gamma([\kappa]^{<\kappa})^+\subseteq \I^{\nu+1}([\kappa]^{<\kappa})^+$, and we may apply the inductive hypothesis to the $H$-list $\vec{B}\restrict H$. Let $A_0$ be the set of all ordinals $\alpha\in H$ such that
\[\exists X\subseteq H\cap\alpha\left[(\forall\xi<\alpha^+\ X\in\I^{f^\kappa_{\nu}(\alpha)}(\Pi^1_\xi(\alpha))^+)\land (X\cup\{\alpha\}\text{ is hom. for }\vec{B}\restrict H)\right].\]
It follows that $H\setminus A_0\in\I^{\nu+1}([\kappa]^{<\kappa})$, which implies that $A_0\in\I^{\nu+1}([\kappa]^{<\kappa})^+$. Fix $\alpha\in A_0$. Since $\alpha\in H$, homogeneity implies that $\vec{B}\restrict (H\cap\alpha)=\vec{B}^\alpha\restrict H$. But, by the definition of $A_0$, and the fact that $\xi_\alpha<\alpha^+$, it follows that there is some $X\in \mathcal P(H\cap\alpha)\cap \I^{f^\kappa_{\nu}}(\Pi^1_{\xi_\alpha}(\alpha))^+$ which is homogeneous for $\vec{B}^\alpha$, a contradiction.
\end{proof}

The following is now immediate from Theorem \ref{theorem_pushing_Baumgartner}.

\begin{corollary}\label{corollary_below_gamma_plus_1_almost_ineffability}
Suppose $\kappa\in\I^{\gamma+1}([\kappa]^{<\kappa})^+$ where $\gamma<\kappa^+$. Then the set
\[\{\alpha<\kappa\st (\forall\xi<\alpha^+)\ \alpha\in\I^{f^\kappa_\gamma(\alpha)}(\Pi^1_\xi(\alpha))^+\}\]
is in the filter $\I^{\gamma+1}([\kappa]^{<\kappa})^*$. \hfill{$\Box$}
\end{corollary}

\subsection{A version of Baumgartner's lemma for the strongly Ramsey ideal}

Recall that in Section \ref{section_introduction}, we 
introduced the strongly Ramsey subset operator $\mathcal{S}$. Next, we show that Baumgartner's lemma can, in a sense, be generalized to the \emph{strongly Ramsey ideal} $\S([\kappa]^{<\kappa})^+$.


\begin{theorem}\label{theorem_baumgartners_lemma_for_the_strongly_ramsey_ideal}
Suppose $S\in\S([\kappa]^{<\kappa})^+$ and $f:[S]^{<\omega}\to\kappa$ is a regressive function. Let
\[A=\{\alpha\in S\st(\exists X\subseteq S\cap \alpha)(\forall\eta<\alpha^+ X\in\Pi^1_\eta(\alpha)^+)\land(\text{$X$ is hom. for $f$})\}.\]
Then, $S\setminus A\in\S([\kappa]^{<\kappa})$.
\end{theorem}

\begin{proof}
Suppose $S\setminus A\in\S([\kappa]^{<\kappa})^+$. Let $M$ be a $\kappa$-model with $S\setminus A,f\in M$ and let $U\subseteq[\kappa]^\kappa$ be a $\kappa$-amenable $M$-normal $M$-ultrafilter such that $S\setminus A\in U$. Let $j:M\to N$ be the usual elementary embedding obtained from $U$ such that $N$ is transitive. 

Since $U$ is $\kappa$-amenable, it follows that for every $B\subseteq\kappa^n\times\kappa$ in $M$, the set $\{\vec{\alpha}\in\kappa^n\st B_{\vec{\alpha}}\in U\}$ is in $M$ where $B_{\vec{\alpha}}=\{\beta<\kappa\st\vec{\alpha}\concat\beta\in B\}$. Thus, we can define the \emph{product ultrafilters} $U^n$ on $P(\kappa^n)^M$ by induction as follows. For $B\subseteq\kappa^n\times\kappa$, we let $B\in U^{n+1}=U^n\times U$ if and only if $B\in M$ and $\{\vec{\alpha}\in\kappa^n\st B_{\vec{\alpha}}\in U\}\in U^n$. For each $n\in\omega\setminus \{0\}$, we let $j_{U^n}:M\to N_{U^n}$ be the ultrapower of $M$ by $U^n$ and note that, it follows from \cite[Proposition 2.32]{MR2710923} that $N_{U^n}$ is well-founded. Furthermore, by \cite[Lemma 2.31]{MR2710923}, we have $j_{U^{n+1}}=j_{j_{U^n}(U)}\circ j_{U^n}$ where $j_{j_{U^n}(U)}$ is the ultrapower of $N_{U^n}$ by $j_{U^n}(U)$ (for more details on product ultrafilters one may consult \cite[Chapter 2]{MR2710923} or \cite{MR2830415}).

For each $n\in\omega\setminus \{0\}$, let $f_n=f\restrict[S]^n$. Since $f$ is regressive, it follows by elementarity that the ordinal $\gamma_1=j_{U^1}(f)(\{\kappa\})$ is less than $\kappa$. Furthermore, for each $n<\omega$ the ordinal $\gamma_{n+2}=j_{U^{n+2}}(f)(\{\kappa,j_U(\kappa),j_{U^2}(\kappa),\ldots,j_{U^{n+1}}(\kappa)\})$ is less than $\kappa$. Fix $n\in\omega\setminus\{0\}$. Let $A_n=\{\vec{\alpha}\in[S]^n\st f_n(\vec{\alpha})=\gamma_n\}$. By \cite[Lemma 2.34]{MR2710923}, there is a $B_n\in U$ such that for all $\beta_1<\cdots<\beta_n$ in $B_n$ we have $(\beta_1,\ldots,\beta_n)\in A_n$, that is, $f_n(\beta_1,\ldots,\beta_n)=\gamma_n$. Hence $B_n$ is homogeneous for $f_n$.

Clearly $B=\bigcap_{n<\omega}B_n$ is homogeneous for $f$, and since $M$ is a $\kappa$-model and $U$ is $M$-normal, it follows that $B\in U$ and hence $\kappa\in j(B)$. Now we have $B\in N$ and furthermore, $N$ thinks that $B$ is homogeneous for $f$. But, since $\kappa\in j(S\setminus A)$, it follows that $N$ thinks that there are no subsets of $S$ which are both $\Pi^1_\eta$-indescribable in $\kappa$ for all $\eta<(\kappa^+)^N$ and homogeneous for $f$. Hence, in $N$, there must be some $\eta<(\kappa^+)^N$ such that $B$ is not $\Pi^1_\eta$-indescribable in $\kappa$. Working in $N$, fix a $\Pi^1_\eta$-sentence $\varphi$ over $V_\kappa$ that is true in $V_\kappa$ such that for all $\alpha\in B$ we have $V_\alpha\models\lnot\varphi\res^\kappa_\alpha$. Since $M$ and $N$ are both $\kappa$-models and since $P(\kappa)^M=P(\kappa)^N$, it follows that $(V_\kappa\models\varphi)^M$ and
\[((\forall\alpha\in B)\ V_\alpha\models\lnot\varphi\res^\kappa_\alpha)^M.\]
Hence by elementarity,
\[((\forall\alpha\in j(B))\ V_\alpha\models\lnot j(\varphi)\res^{j(\kappa)}_\alpha)^N.\]
But this is a contradiction because $\kappa\in j(B)$ and $(V_\kappa\models\varphi)^N$ where $j(\varphi)\res^{j(\kappa)}_\kappa=\varphi$ (see the proof of Lemma \ref{lemma_restrictionvsj}).
\end{proof}

\section{Indescribability from homogeneity}\label{section:indescribability_from_homogeneity}

Extending \cite[Lemma 7.1]{MR0384553} and \cite[Lemma 5.1]{MR4206111}, we show that for all $\xi<\kappa^+$, $S\in\I(\Pi^1_\xi(\kappa))^+$ implies $S\in\Pi^1_{\xi+2}(\kappa)^+$. Let us note that the following lemma has precursors in the work of Welch et al. (see \cite[Corollary 3.24]{MR2817562} and \cite{brickhill-welch}).

\begin{lemma}\label{lemma_indescribability_from_homogeneity}
Suppose $S\subseteq\kappa$, $\xi<\kappa^+$, and for every $S$-list $\vec{S}$, there is a set $H\in \mathcal P(S)\cap\bigcap_{\zeta\in\{-1\}\cup\xi}\Pi^1_\zeta(\kappa)^+$ that is homogeneous for $\vec{S}$. Then, $S$ is a $\Pi^1_{\xi+1}$-indescribable subset of $\kappa$.
\end{lemma}

\begin{proof}
The case in which $\xi<\kappa$ is handled by \cite[Lemma 2.20]{MR4206111}. The case in which $\kappa<\xi<\kappa^+$ and $\xi$ is a successor ordinal is similar (see the corresponding case in \cite[Lemma 2.20]{MR4206111}), and is thus left to the reader.

Suppose $\kappa\leq\xi<\kappa^+$, $\xi$ is a limit ordinal, and every $S$-list has a homogeneous set $H\in \mathcal P(S)\cap\bigcap_{\zeta<\xi}\Pi^1_\zeta(\kappa)^+$. Suppose for a contradiction that $S$ is not $\Pi^1_{\xi+1}$-indescribable. Let \[\varphi=\forall X\left(\bigvee_{\zeta<\xi}\psi_\zeta\right)\] be $\Pi^1_{\xi+1}$ over $V_\kappa$, such that $V_\kappa\models\varphi$, and such that for all $\alpha\in S$, we have $V_\alpha\not\models\varphi\res^\kappa_\alpha$ whenever $\varphi\res^\kappa_\alpha$ is defined. Fix a bijection $b:V_\kappa\to\kappa$, let 
\[C=\{\alpha<\kappa\st b\restrict\alpha:V_\alpha\to\alpha\text{ is a bijection and }\varphi\res^\kappa_\alpha\text{ is  defined}\},\]
and note that $C$ is a club subset of $\kappa$ by Lemma \ref{lemma_restriction_is_nice}. For each $\alpha\in S\cap C$, the sentence $\varphi\res^\kappa_\alpha$ thus is $\Pi^1_{f^\kappa_{\xi+1}(\alpha)}$ over $V_\alpha$, and hence $V_\alpha\models\lnot\varphi\res^\kappa_\alpha$, where
\[\lnot\varphi\res^\kappa_\alpha=\exists X\left(\bigwedge_{\zeta\in f^\kappa_\xi(\alpha)}\lnot\psi_{(\pi^\kappa_{\xi,\alpha})^{-1}(\zeta)}\res^\kappa_\alpha\right).\] 
For each $\alpha\in S\cap C$, let $T_\alpha\subseteq V_\alpha$ be such that 
\begin{align}V_\alpha\models\bigwedge_{\zeta\in f^\kappa_\xi(\alpha)}(\lnot\psi_{(\pi^\kappa_{\xi,\alpha})^{-1}(\zeta)}\res^\kappa_\alpha)(T_\alpha).\label{equation_ind_from_hom}
\end{align}
For $\alpha\in S\setminus C$ let $T_\alpha\subseteq\alpha$ be arbitrary. Now we define an $S$-list $\vec{S}=\<S_\alpha\st\alpha\in S\>$ where $S_\alpha=b[T_\alpha]$ for all $\alpha\in S$. Let $H\in \mathcal P(S)\cap\bigcap_{\zeta<\xi}\Pi^1_\zeta(\kappa)^+$ be ho\-mo\-ge\-ne\-ous for $\vec{S}$. Notice that $H\cap C\in\bigcap_{\zeta<\xi}\Pi^1_\zeta(\kappa)^+$, let $R=\bigcup_{\alpha\in H\cap C}S_\alpha$, and let $T=b^{-1}[R]$.
Since $V_\kappa\models\varphi$, we have $V_\kappa\models\psi_\zeta(T)$ for some fixed $\zeta<\xi$. Since the set
\[H\cap C\cap\{\alpha<\kappa\st \zeta\in F^\kappa_\xi(\alpha)\}\] is $\Pi^1_\zeta$-indescribable in $\kappa$, it follows that for some $\alpha\in H\cap C$ with $\zeta\in F^\kappa_\xi(\alpha)$ we have $V_\alpha\models(\psi_\zeta\res^\kappa_\alpha)(T\cap V_\alpha)$. By homogeneity, we have $R\cap\alpha=S_\alpha$, and hence $T\cap V_\alpha=T_\alpha$. This implies that $V_\alpha \models(\psi_\zeta\res^\kappa_\alpha)(T_\alpha)$, but this contradicts (\ref{equation_ind_from_hom}).
\end{proof}

It was shown in \cite[Proposition 3.8]{CodyHigherIndescribability} that measurable cardinals are $\Pi^1_\xi$-indescribable for every $\xi<\kappa^+$. We want to make use of Lemma \ref{lemma_indescribability_from_homogeneity} in order to provide a better upper bound. Recall that a cardinal $\kappa$ is \emph{completely ineffable} if there is a collection $\mathcal S$ of stationary subsets of $\kappa$ that is closed under the taking of supersets (such collections are called a \emph{stationary class}), such that whenever $S\in\mathcal S$ and $\vec S$ is an $S$-list, then there is a set $H\in\mathcal S$ that is homogeneous for $\vec S$.

\begin{proposition}
  If $\kappa$ is completely ineffable, then $\kappa$ is $\Pi^1_\xi$-indescribable for every $\xi<\kappa^+$.
\end{proposition}
\begin{proof}
  Let $\mathcal T$ be the union of all stationary classes witnessing that $\kappa$ is completely ineffable. It is easy to see that $\mathcal T$ itself is a stationary class witnessing that $\kappa$ is completely ineffable, and also that $I=\mathcal P(\kappa)\setminus\mathcal T$ is an ideal on $\kappa$ -- in fact, $I$ is what is called the \emph{completely ineffable ideal} on $\kappa$, as defined in \cite{MR918427}. Note that by the very definition of $\I$, we see that $\I(I)=I$. Using Lemma \ref{lemma_indescribability_from_homogeneity}, and recalling that $\Pi^1_0(\kappa)\subseteq I$ is the nonstationary ideal on $\kappa$, a straightforward induction now yields $\kappa$ to be $\Pi^1_\xi$-indescribable for every $\xi<\kappa^+$.
\end{proof}

\section{Some properties of the ineffability and the Ramsey operator}\label{section:basic}

In this section, we will provide two lemmas about the ineffability and the Ramsey operator which will be required later on, but which should also be of independent interest. For the Ramsey operator, when $\gamma$ and $\xi$ are both less than $\kappa$, these are due to the first author in \cite[Lemma 3.1 and Lemma 3.2]{MR4206111}.

\begin{lemma}\label{lemma_pos_union_of_pos_sets_is_pos}
Let $\O\in\{\I,\R\}$. Suppose $\kappa$ is a cardinal, $\gamma<\kappa^+$, $\xi\in\{-1\}\cup\kappa^+$, $S\in\O^\gamma(\Pi^1_\xi(\kappa))^+$, and for each $\alpha\in S$, let $S_\alpha\in\O^{f^\kappa_\gamma(\alpha)}(\Pi^1_{f^\kappa_\xi(\alpha)}(\alpha))^+$. Then $\bigcup_{\alpha\in S}S_\alpha\in\O^\gamma(\Pi^1_\xi(\kappa))^+$.
\end{lemma}

\begin{proof}
Let us assume that $\O=\I$; when $\O=\R$ the proof is essentially the same, only one must replace lists by regressive functions. We proceed by induction on $\gamma$. Suppose $\gamma=0$, fix $\xi<\kappa^+$, $S\in\Pi^1_\xi(\kappa)^+$ and let $S_\alpha\in\Pi^1_{f^\kappa_\xi(\alpha)}(\alpha)^+$ for all $\alpha\in S$. Fix a $\Pi^1_\xi$ sentence $\varphi$ over $V_\kappa$ such that $V_\kappa\models\varphi$. By Lemma \ref{lemma_restriction_is_nice}, the set
\[C=\{\alpha<\kappa\st\varphi\res^\kappa_\alpha\text{ is defined, and hence $\Pi^1_{f^\kappa_\xi(\alpha)}$ over $V_\alpha$}\}\]
is in the club filter on $\kappa$. By Proposition \ref{proposition_double_restriction}, there is a club subset $D$ of $\kappa$ such that for all regular uncountable $\alpha\in D$, the set $D_\alpha$ of all ordinals $\beta<\alpha$ for which $(\varphi\res^\kappa_\alpha)\res^\alpha_\beta$ is defined and $(\varphi\res^\kappa_\alpha)\res^\alpha_\beta=\varphi\res^\kappa_\beta$ is in the club filter on $\alpha$. Since $S\cap C\cap D\in\Pi^1_\xi(\kappa)^+$, we may fix an $\alpha\in S\cap C\cap D$ such that $V_\alpha\models\varphi\res^\kappa_\alpha$. Now, since $S_\alpha\cap D_\alpha\in\Pi^1_{f^\kappa_\xi(\alpha)}(\alpha)^+$ and $\varphi\res^\kappa_\alpha$ is $\Pi^1_{f^\kappa_\xi(\alpha)}$ over $V_\alpha$, we may fix $\beta\in S_\alpha\cap D_\alpha$ such that $V_\beta\models(\varphi\res^\kappa_\alpha)\res^\alpha_\beta$. Since $\beta\in D_\alpha$ implies $(\varphi\res^\kappa_\alpha)\res^\alpha_\beta=\varphi\res^\kappa_\beta$, we have $V_\beta\models\varphi\res^\kappa_\beta$. Thus, $\bigcup_{\alpha\in S}S_\alpha$ is a $\Pi^1_\xi$-indescribable subset of $\kappa$.

Suppose $\gamma<\kappa^+$ is a limit ordinal. Fix $\xi<\kappa^+$, $S\in\I^\gamma(\Pi^1_\xi(\kappa))^+$ and let $S_\alpha\in \I^{f^\kappa_\gamma(\alpha)}(\Pi^1_{f^\kappa_\xi(\alpha)}(\alpha))^+$ for all $\alpha\in S$. It suffices to show that $\bigcup_{\alpha\in S}S_\alpha\in\I^\delta(\Pi^1_\xi(\kappa))^+$ for all $\delta<\gamma$. Fix $\delta<\gamma$. Since $\delta<\gamma$ in any generic ultrapower $\Ult$ obtained by forcing with $P(\kappa)/\NS_\kappa$, the set
\[C=\{\alpha<\kappa\st f^\kappa_\delta(\alpha)<f^\kappa_\gamma(\alpha)\}\]
is in the club filter on $\kappa$. Thus, $S\cap C\in\I^\delta(\Pi^1_\xi(\kappa))^+$, and for each $\alpha\in S\cap C$, we have $S_\alpha\in\I^{f^\kappa_\delta(\alpha)}(\Pi^1_{f^\kappa_\xi(\alpha)})^+$. By our inductive hypothesis, we have $\bigcup_{\alpha\in S\cap C}S_\alpha\in\I^\delta(\Pi^1_\xi(\kappa))^+$. Since $\bigcup_{\alpha\in S\cap C}S_\alpha\subseteq\bigcup_{\alpha\in S}S_\alpha$, we thus see that $\bigcup_{\alpha\in S}S_\alpha\in\I^\gamma(\Pi^1_\xi(\kappa))^+$.

Suppose $\gamma=\delta+1$ is a successor ordinal. Fix $\xi<\kappa^+$, $S\in\I^{\delta+1}(\Pi^1_\xi(\kappa))^+$ and let $S_\alpha\in \I^{f^\kappa_{\delta+1}(\alpha)}(\Pi^1_{f^\kappa_\xi(\alpha)})^+$ for each $\alpha\in S$. Let $T=\bigcup_{\alpha\in S}S_\alpha$. Fix a $T$-list $\vec{T}=\<T_\alpha\st\alpha\in T\>$. We must show that there is a homogeneous set for $\vec{T}$ in $\I^\delta(\Pi^1_\xi(\kappa))^+$. By Lemma \ref{lemma_can}, the set 
\[C=\{\alpha<\kappa\st f^\kappa_{\delta+1}(\alpha)=f^\kappa_\delta(\alpha)+1\}\] is in the club filter on $\kappa$. Thus $S\cap C\in\I^{\delta+1}(\Pi^1_\xi(\kappa))^+$. For each $\alpha\in S\cap C$, the $S_\alpha$-list $\vec{T}\restrict S_\alpha$ has a homogeneous set $H_\alpha\in \mathcal P(S_\alpha)\cap\Pi^1_{f^\kappa_\delta(\alpha)}(\alpha)^+$. Let $H\in\I^\delta(\Pi^1_\xi(\kappa))^+$ be homogeneous for the $(S\cap C)$-list $\<H_\alpha\mid\alpha\in S\cap C\>$. By our inductive hypothesis, $\bigcup_{\alpha\in H}H_\alpha\in\I^\delta(\Pi^1_\xi(\kappa))^+$, and it is easy to see that this set is homogeneous for $\vec{T}$.
\end{proof}

\begin{lemma}\label{lemma_set_of_nons_is_positive}
Let $\O\in\{\I,\R\}$. Suppose $\kappa$ is a cardinal, $\gamma<\kappa^+$ and $\xi\in\{-1\}\cup\kappa^+$. If $\kappa\in\O^\gamma(\Pi^1_\xi(\kappa))^+$, then the set
\[S_\kappa=\{\alpha<\kappa\st\alpha\in\O^{f^\kappa_\gamma(\alpha)}(\Pi^1_{f^\kappa_\xi(\alpha)}(\alpha))\}\]
is in $\O^\gamma(\Pi^1_\xi(\kappa))^+$.
\end{lemma}

\begin{proof}
Assume for a contradiction that the statement of the lemma does not hold true, and let $\kappa$ be the least counterexample: the least cardinal for which there are $\gamma<\kappa^+$ and $\xi\in\{-1\}\cup\kappa^+$ such that $\kappa\in\O^\gamma(\Pi^1_\xi(\kappa))^+$ and $S:=S_\kappa\in\O^\gamma(\Pi^1_\xi(\kappa))$. Then, $\kappa\setminus S\in\O^\gamma(\Pi^1_\xi(\kappa))^+$. For each $\alpha\in\kappa\setminus S$, we have $\alpha\in\O^{f^\kappa_\gamma(\alpha)}(\Pi^1_{f^\kappa_\xi(\alpha)}(\alpha))^+$, and by the minimality of $\kappa$, the set $S_\alpha=S\cap\alpha$ is in $\O^{f^\kappa_\gamma(\alpha)}(\Pi^1_{f^\kappa_\xi(\alpha)}(\alpha))^+$. Thus, by Lemma \ref{lemma_pos_union_of_pos_sets_is_pos}, the set $S=\bigcup_{\alpha\in\kappa\setminus S}S_\alpha$ is in $\O^\gamma(\Pi^1_\xi(\kappa))^+$, a contradiction.
\end{proof}

Next, we provide a result for the strongly Ramsey ideal which is analogous to the base case of Lemma \ref{lemma_set_of_nons_is_positive}. This result follows from more general results in \cite[Lemma 14.2]{MR4156888} (with the core argument being \cite[Lemma 9.15]{MR4156888}), however we would like to provide a proof for the particular case of strongly Ramsey cardinals, also in order to allow for the discussion of possible generalizations that follows in the remark below.

\begin{lemma}[Holy-L\"ucke]\label{lemma_set_of_nons_for_strongly_ramsey_ideal}
For every cardinal $\kappa$, if $\kappa\in\S([\kappa]^{<\kappa})^+$, then the set
\[T=\{\alpha<\kappa\st\alpha\in\S([\alpha]^{<\alpha})\}\]
is in $\S([\kappa]^{<\kappa})^+$.
\end{lemma}
\begin{proof}
Suppose the result is false and let $\kappa$ be the least counterexample. Then $\kappa$ is strongly Ramsey and $T\in\S([\kappa]^{<\kappa})$. This implies $\kappa\setminus T\in\S([\kappa]^{<\kappa})^*$ and hence there is an $A_T\subseteq\kappa$ such that whenever $M$ is a $\kappa$-model with $A_T,\kappa\setminus T\in M$ and whenever $U\subseteq[\kappa]^\kappa$ is a $\kappa$-amenable $M$-normal $M$-ultrafilter on $\kappa$, it must follow that $\kappa\setminus T\in U$. Let $j:M\to N$ be the ultrapower embedding obtained from $U$, and notice that $\kappa\in j(\kappa\setminus T)$ and hence $\kappa$ is strongly Ramsey in $N$.

By our assumption on $\kappa$, it follows that for all $\alpha<\kappa$, if $\alpha$ is strongly Ramsey then the set
\[T\cap\alpha=\{\beta<\alpha\st\beta\in\S([\beta]^{<\beta})\}\]
is in $\S([\alpha]^{<\alpha})^+$. Since $M$ is a $\kappa$-model, this statement also holds in $M$. So, since $\kappa$ is strongly Ramsey in $N$, it follows by elementarity that in $N$, the set $j(T)\cap\kappa=T$ is in $(\S([\kappa]^{<\kappa})^+)^N$. Working in $N$, we let $\bar M$ be a $\kappa$-model with $A_T,T\in \bar M$ and we let $\bar U$ be a $\kappa$-amenable $\bar M$-normal $\bar M$-ultrafilter on $\kappa$ with $\bar U\subseteq ([\kappa]^\kappa)^N$ and $T\in \bar U$. Since $N$ is a $\kappa$-model, it follows that in $V$ the set $\bar M$ is a $\kappa$-model with $A_T,T\in \bar M$, $\bar U$ is a $\kappa$-amenable $\bar M$-normal $\bar M$-ultrafilter on $\kappa$ with $\bar U\subseteq[\kappa]^\kappa$, and $T\in U$. This contradicts the definition of $A_T$.
\end{proof}

\begin{remark}
Let us note that we do not know whether a version of Lemma \ref{lemma_set_of_nons_for_strongly_ramsey_ideal} holds for the ideal $\S^2([\kappa]^{<\kappa})$. Suppose $\kappa\in\S^2([\kappa]^{<\kappa})^+$ and let $T=\{\alpha<\kappa\st\alpha\in\S^2([\alpha]^{<\alpha})\}$. Does it follow that $T\in\S^2([\kappa]^{<\kappa})^+$? If we try generalize the proof of Lemma \ref{lemma_set_of_nons_for_strongly_ramsey_ideal} to this situation, we would like to show that if a $\kappa$-model $N$ thinks that $\bar M$ is a $\kappa$-model and $\bar U$ is a $\kappa$-amenable $\bar M$-normal $\bar M$-ultrafilter with $\bar U\subseteq(\S([\kappa]^{<\kappa})^+)^N$, then it is the case that in $V$ we have $\bar U\subseteq(\S([\kappa]^{<\kappa})^+)^V$. However, we do not see how to prove this. One would want to show that $(\S([\kappa]^{<\kappa})^+)^N\subseteq(\S([\kappa]^{<\kappa})^+)^V$. But this seems to be problematic because $P(\kappa)^N\subsetneq P(\kappa)^V$.
\end{remark}

\section{Expressibility results}\label{section:expressibility}

First, let us recall an expressibility result for higher indescribability due to the first author, which extends results of Bagaria from \cite{MR3894041}.

\begin{theorem}[{\cite[Theorem 5.8]{CodyHigherIndescribability}}]\label{theorem_expressing_indescribability}
Suppose $\kappa>\omega$ is regular and $\xi<\kappa^+$. Then, there is a $\Pi^1_{\xi+1}$ formula $\Phi$ over $V_\kappa$ and a club $C\subseteq\kappa$ such that for all $S\subseteq\kappa$ we have
\[\text{$S$ is a $\Pi^1_\xi$-indescribable subset of $\kappa$ if and only if $V_\kappa\models\Phi(S)$}\]
and for all regular $\alpha\in C$, we have
\[\text{$S\cap\alpha$ is a $\Pi^1_{f^\kappa_\xi(\alpha)}$-indescribable subset of $\alpha$ if and only if $V_\alpha\models\Phi(S)\res^\kappa_\alpha$}.\]
\end{theorem}

Note that, within our usual generic ultrapower setup, using Lemma \ref{lemma_restrictionvsj}, the existence of a club $C$ as for the second statement of Theorem \ref{theorem_expressing_indescribability} above is equivalent to its first statement holding in the generic ultrapower $\Ult$. This could be used to extract a fairly simple proof of the second statement from the original proof of the first statement that is provided in \cite{CodyHigherIndescribability}. Since doing this in detail would involve going through quite a lot of material from \cite{CodyHigherIndescribability} however, we will leave this task to the interested reader.

We will next need an easy lemma, whose proof, via a standard closing-off argument, is left to the reader as well.

\begin{lemma}\label{lemma_closing_off}
Suppose $\kappa$ is a regular cardinal and $\gamma,\gamma'<\kappa^+$ are ordinals such that $\gamma\leq\gamma'$. If $f:\gamma\to\gamma'$ is any function then the set
\[\{\alpha<\kappa\st f[F^\kappa_\gamma(\alpha)]\subseteq F^\kappa_{\gamma'}(\alpha)\}\]
is in the club filter on $\kappa$.
\end{lemma}


Building on Theorem \ref{theorem_expressing_indescribability} and \cite[Lemma 5.1]{MR4206111}, we obtain the following.

\begin{lemma}\label{lemma_complexity}
Suppose $\kappa$ is a regular cardinal, and $\gamma<\kappa^+$, and $\xi\in\{-1\}\cup\kappa^+$. Let $\mathcal{O}\in\{\mathcal{I},\mathcal{R}\}$ be either the ineffability operator or the Ramsey operator. Then, there is a $\Pi^1_{\xi+1+2\gamma}$ formula $\Theta^\kappa_{\gamma,\xi}$ over $V_\kappa$ and a club subset $C^\kappa_{\gamma,\xi}$ of $\kappa$ such that for all $S\subseteq\kappa$ we have
\[S\in\O^\gamma(\Pi^1_\xi(\kappa))^+\text{ if and only if }V_\kappa\models\Theta^\kappa_{\gamma,\xi}(S)\]
and for all regular cardinals $\alpha\in C^\kappa_{\gamma,\xi}$ we have
\[S\cap\alpha\in \O^{f^\kappa_\gamma(\alpha)}(\Pi^1_{f^\kappa_\xi(\alpha)}(\alpha))^+\text{ if and only if }V_\alpha\models\Theta^\kappa_{\gamma,\xi}(S)\res^\kappa_\alpha.\]
\end{lemma}
\begin{proof}
For the final statement, we will make use of our usual generic ultrapower setup once again: Using Lemma \ref{lemma_restrictionvsj}, it easily follows that the existence of a club as for the second statement of Lemma \ref{lemma_complexity} above is equivalent to its first statement holding in any generic ultrapower $\Ult$ obtained by forcing with $P(\kappa)/\NS_\kappa$.\footnote{Of course, when we refer to the operators $\I$ or $\R$ in $\Ult$, these should be the ineffability or the Ramsey operator as defined in $\Ult$, respectively. Also, notice that since $j(\Theta^\kappa_{\gamma,\xi}(X))\res^{j(\kappa)}_\kappa=\Theta^\kappa_{\gamma,\xi}(X)\in \Ult$, it follows that the statement $(S\in\O^\gamma(\Pi^1_\xi(\kappa))^+\text{ if and only if }V_\kappa\models\Theta^\kappa_{\gamma,\xi}(S))^{\text{Ult}}$ makes sense.} We thus proceed by induction on $\gamma<\kappa^+$ to verify the first statement in $V$ and in $\Ult$ simultaneously. Let us consider the case in which $\mathcal{O}=\mathcal{I}$; the case in which $\mathcal{O}=\mathcal{R}$ is similar. If $\gamma=0$, then for all $\xi<\kappa^+$, we have $\I^\gamma(\Pi^1_\xi(\kappa))^+=\Pi^1_\xi(\kappa)^+$, and the result follows directly from Theorem \ref{theorem_expressing_indescribability} and the comments made afterwards (regarding the case of the generic ultrapower $\Ult$). 

Suppose $\gamma=\delta+1$, and that the result holds for $\delta$. Fix $\xi<\kappa^+$. Then, there is a $\Pi^1_{\xi+1+2\delta}$-formula $\Theta^\kappa_{\delta,\xi}$ over $V_\kappa$ such that both in $V$ and in $\Ult$, for all $S\subseteq\kappa$, we have
\begin{align}S\in\I^{\delta}(\Pi^1_\xi(\kappa))^+\text{ if and only if }V_\kappa\models\Theta^\kappa_{\delta,\xi}(S).\label{equation_gamma_less_than_omega_at_kappa}
\end{align}
We simply define $\Theta^\kappa_{\gamma,\xi}$ to be the $\Pi^1_{\xi+1+2\gamma}$-formula over $V_\kappa$ which asserts that every $X$-list has a homogeneous set $Y$ such that $\Theta^\kappa_{\delta,\xi}(Y)$ holds. It is clear that this formula is as desired both in $V$ and in $\Ult$.

Suppose now that $\gamma$ is a limit ordinal, and that the result holds for all ordinals $\delta<\gamma$. By definition, we have $X\in\I^\gamma(\Pi^1_\xi(\kappa))^+$ if and only if $X\in\I^\delta(\Pi^1_\xi(\kappa))^+$ for all $\delta<\gamma$. Note that $\xi+1+2\gamma=\xi+\gamma$ in this case. We define a $\Pi^1_{\xi+\gamma}$-formula 
\[\Theta^\kappa_{\gamma,\xi}=\bigwedge_{\zeta<\xi+\gamma}\psi^\kappa_\zeta\]
as follows. For each $\zeta<\xi+\gamma$, if it exists, define $\delta_\zeta$ to be the greatest ordinal $\delta<\gamma$ such that $\xi+1+2\delta\leq\zeta$ and let $\psi^\kappa_\zeta=\Theta^\kappa_{\delta_\zeta,\xi}$. Otherwise, let $\psi^\kappa_\zeta$ be the formula ``$0=0$''. 
Since the sequence $\vec{\delta}=\<\delta_\zeta\st\zeta<\xi+\gamma\>$ is cofinal in $\gamma$, it follows that both in $V$ and in $\Ult$, for all $S\subseteq\kappa$,
\[S\in\I^\gamma(\Pi^1_\xi(\kappa))^+\text{ if and only if }V_\kappa\models\Theta^\kappa_{\gamma,\xi}(S).\]
\end{proof}

\section{A framework for large cardinal operators}\label{section_generalized operators}

In this section, we review a framework for large cardinal operators that was introduced by the second author \cite{HolyLCOandEE}, which in particular fits the ineffability operator $\mathcal I$, the Ramsey operator $\mathcal R$, and the strongly Ramsey subset operator $\mathcal S$ (the latter was denoted as $\mathbf T_{\mathrm{cl}}$ in \cite{HolyLCOandEE}). This framework builds on statements about the existence of certain ultrafilters for small models of set theory, and is itself based on a framework for the characterization of large cardinal ideals that was introduced in \cite{MR4156888}. In the present paper, we apply this framework from \cite{HolyLCOandEE}, verifying a number of results on the relationship between higher indescribability and large cardinal operators in a uniform way. In particular, we thus obtain a number of new results on the relationship between higher indescribability and the operators $\I$ and $\R$, and also $\S$. For readers only interested in these examples, our framework is still useful, for it provides uniform arguments that work for each of these operators. We will also mention (see Remark \ref{remark_more_examples}) two additional operators, introduced by the second author \cite{HolyLCOandEE}, that fit into this framework: the $\mathbf T_\omega^\kappa$-Ramsey subset operator~$\mathbf T$ that is connected to the notion of $\mathbf T_\omega^\kappa$-Ramsey cardinals introduced in \cite{MR4156888}, and the $\mathbf{wf}^\kappa_\omega$-Ramsey subset operator $\mathbf{wf}$ that is connected to the notion of weakly Ramsey cardinals from \cite{MR2830415}, to which our results thus apply.

Let us assume throughout this section that $\kappa$ denotes an inaccessible cardinal, and that $I$ denotes an ideal on $\kappa$. 
Recall that an $M$-ultrafilter $U$ on $\kappa$ is \emph{$\kappa$-amenable for $M$} if whenever $\mathcal A\in M$ is a $\kappa$-sized collection of subsets of $\kappa$ in $M$, then $\mathcal A\cap U\in M$. We next provide the definition of the \emph{model version} $\I_{mod}$ of the ineffability operator, as introduced in \cite{HolyLCOandEE}.

\begin{definition}
\begin{itemize}
  \item  For any $y\subseteq\kappa$, we first define the local instance of $\I_{mod}$ at $y$, by letting $x\in\I_{mod}^y(I)^+$ if there is a transitive weak $\kappa$-model $M$ with $y\in M$, and an $M$-ultrafilter $U$ on $\kappa$ with $x\in U$, such that every diagonal intersection of $U$ is in $I^+$  -- we abbreviate this latter property of $U$ and of $I$ by stating that $\Delta U\in I^+$.\footnote{\label{footnote:diagonalintersections}Since permuting the input of a diagonal intersection only changes its output by a non-stationary set (see \cite[Lemma 1.3.3]{MR0460120}), if $I\supseteq\NS_\kappa$, rather than requiring that every diagonal intersection of $U$ be in $I^+$, it equivalently suffices to require one (arbitrary) diagonal intersection of $U$ to be in $I^+$.}
   \item We let $\I_{mod}(I)^+=\bigcap_{y\subseteq\kappa}\I_{mod}^y(I)^+$.
\end{itemize}
\end{definition}

\begin{proposition}[\cite{HolyLCOandEE} Proposition 2.5]\label{proposition:ineffableideal2}
  Let $I\supseteq\NS_\kappa$ be an ideal on $\kappa$. Then, $\I_{mod}(I)=\mathcal I(I)$. 
\end{proposition}

We also provide the \emph{model version} of the Ramsey operator from \cite{HolyLCOandEE}.

\begin{definition}
  \begin{itemize}
    \item For any $y\subseteq\kappa$, we first define the local instance of $\R_{mod}$ at $y$, by letting $x\in\R_{mod}^y(I)^+$ if there is a transitive weak $\kappa$-model $M$ with $y\in M$, and an $M$-normal $M$-ultrafilter $U$ on $\kappa$ with $x\in U$ that is $\kappa$-amenable for $M$, such that every countable intersection of elements of $U$ is in $I^+$.
    \item We let $\R_{mod}(I)^+=\bigcap_{y\subseteq\kappa}\R_{mod}^y(I)^+$.
   \end{itemize}
\end{definition}

The Ramsey operator and its model version were shown to be equivalent in \cite{MR2817562}. See also \cite{HolyLCOandEE}.

\begin{theorem}[{Sharpe and Welch \cite{MR2817562}}]\label{theorem:ramseymodels}
  For any ideal $I$, \[\R_{mod}(I)=\mathcal R(I).\]
\end{theorem}

Taking the above characterizations of the ineffability and of the Ramsey operator as an inspiration, a framework for large cardinal operators was developed in \cite{HolyLCOandEE}, which we would now like to review.

\begin{definition}\label{definition:abstractoperator}
  Let $\Psi(M,U)$ and $\Omega(U,I)$ be parameter-free first order formulae such that $\ZFC$ proves that for any ideal $I$ on a regular uncountable cardinal $\kappa$, any transitive weak $\kappa$-model $M$ and any $M$-ultrafilter $U$ on $\kappa$, 
  \begin{itemize}
    \item $\Omega(U,I)$ implies that $U\subseteq I^+$, and
    \item for any ideal $J$ on $\kappa$, $\left[I\supseteq J\,\land\,\Omega(U,I)\right]\to\Omega(U,J)$.
  \end{itemize}
  Let us say that a pair of formulas $\langle\Psi,\Omega\rangle$ satisfying the above is \emph{regular}.
  
  We define an ideal operator $\mathfrak O\Psi\Omega$ as follows. For any ideal $I$ on $\kappa$ and $y\subseteq\kappa$, we first define a local instance by letting
  \begin{itemize}
    \item $x\in\mathfrak O\Psi\Omega^y(I)^+$ if there exists a transitive weak $\kappa$-model $M$ with $y\in M$ and an $M$-ultrafilter $U$ on $\kappa$ with $x\in U$ such that $\Psi(M,U)$ and $\Omega(U,I)$ hold, and we let
    \item $\mathfrak O\Psi\Omega(I)^+=\bigcap_{y\subseteq\kappa}\mathfrak O\Psi\Omega^y(I)^+$.
  \end{itemize}
\end{definition}

Let us remark that, since we assume $\kappa$ to be inaccessible, we could additionally require that $M\supseteq V_\kappa$ in the above, for given any $y\subseteq\kappa$, we can easily find $y'\subseteq\kappa$ such that $y'\in M$ implies both that $y\in M$ and that $V_\kappa\subseteq M$.

\medskip

Let us check how the examples we saw so far fit into these schemes:

\begin{itemize}
  \item If $\Psi(M,U)$ is trivial, and $\Omega(U,I)$ denotes the property that $\Delta U{\in}I^+$, then $\mathfrak O\Psi\Omega$ is the model version $\I_{mod}$ of the ineffability operator.
  \item If $\Psi(M,U)$ denotes the property that $U$ is $M$-normal and $\kappa$-amenable for $M$, and $\Omega(U,I)$ denotes the property that every countable intersection of elements of $U$ is in $I^+$, then $\mathfrak O\Psi\Omega$ is (the model version $\R_{mod}$ of) the Ramsey operator.
  \item If $\Psi(M,U)$ denotes the property that $M$ is closed under ${<}\kappa$-sequences, $U$ is $M$-normal and $U$ is $\kappa$-amenable for $M$, and $\Omega(U,I)$ denotes the property $U\subseteq I^+$, then $\mathfrak O\Psi\Omega$ is the strongly Ramsey subset operator $\S$.
\end{itemize}

\begin{proposition}\cite[Proposition 10.2]{HolyLCOandEE}\label{proposition:abstractbasic}
  Assume that $\langle\Psi,\Omega\rangle$ is regular, and that $I\supseteq J$ are ideals on $\kappa$. Then, the following hold.
  \begin{itemize}
    \item $\mathfrak O\Psi\Omega(I)\supseteq I$ is an ideal on $\kappa$.
    \item $\mathfrak O\Psi\Omega(I)\supseteq\mathfrak O\Psi\Omega(J)$.
    \item If for any transitive weak $\kappa$-model $M$ and any $M$-ultrafilter $U$ on $\kappa$, the conjunction $\Psi(M,U)\,\land\,\Omega(U,I)$ implies that $U$ is $M$-normal, then $\mathfrak O\Psi\Omega(I)$ is normal.
    \item In particular, if $I\supseteq\NS_\kappa$, then $\Delta U\in I^+$ implies that $U$ is $M$-normal.
    \item If $\langle \Psi',\Omega'\rangle$ is regular as well, and $\Psi'(M,U)\land\Omega'(U,I)$ implies $\Psi(M,U)\land\Omega(U,I)$ for any transitive weak $\kappa$-model $M$ and any $M$-ultrafilter $U$ on $\kappa$, then $\mathfrak O\Psi'\Omega'(I)\supseteq\mathfrak O\Psi\Omega(I)$.
  \end{itemize}
\end{proposition}

A crucial property of ideal operators is \emph{ineffability}, as introduced in \cite{HolyLCOandEE}.

\begin{definition}
Let $\langle\Psi,\Omega\rangle$ be a pair of formulas, and let $\mathcal O$ be an ideal operator.
\begin{itemize}
   \item The pair $\langle\Psi,\Omega\rangle$ is \emph{ineffable} in case $\ZFC$ proves that for any ideal $I$ on a regular uncountable cardinal $\kappa$, any transitive weak $\kappa$-model $M$ and any $M$-ultrafilter $U$ on $\kappa$, $\Psi(M,U)\,\land\,\Omega(U,I)$ implies that for every $A\in U$, every $A$-list $\vec a\in M$ has a homogeneous set in $I^+$.
    \item The operator $\mathcal O$ is \emph{ineffable} in case $\ZFC$ proves that for any ideal $I$ on a regular uncountable cardinal $\kappa$, whenever $A\in\mathcal O(I)^+$ and $\vec a$ is an $A$-list, then $\vec a$ has a homogeneous set in $I^+$.
  \end{itemize}
\end{definition}
Note that by the above, the ineffability operator $\I$ is ineffable. But also, if $\mathcal O$ can be characterized to be of the form $\mathcal O=\mathfrak O\Psi\Omega$ for some ineffable pair of formulas $\langle\Psi,\Omega\rangle$, then $\mathcal O$ is ineffable.

\begin{observation}\label{observation:regularity}\cite[Observation 10.4]{HolyLCOandEE}
  Let $\langle\Psi,\Omega\rangle$ be regular, and let $\mathcal O$ be the operator $\mathfrak O\Psi\Omega$. Then,
  \begin{itemize}
    \item If $\Psi(M,U)$ $\ZFC$-provably implies that $U$ is $\kappa$-amenable for $M$ and contains all club subsets of $\kappa$ in $M$ as elements, then $\mathcal O$ is ineffable.
    \item If $\Omega(U,I)$ $\ZFC$-provably implies that $\Delta U\in I^+$, then $\mathcal O$ is ineffable.
    \item If $\mathcal O$ is ineffable, then for any ideal $I$ on a regular uncountable cardinal $\kappa$, $\mathcal O(I)\supseteq\mathcal I(I)\supseteq\NS_\kappa$.
    \item If $\ZFC$ proves that for any ideal $I$ on a regular and uncountable cardinal $\kappa$, $\mathcal O(I)\supseteq\mathcal I(I)$, then $\mathcal O$ is ineffable.
  \end{itemize}
\end{observation}

In particular, the above implies that the operators $\mathcal R$ and $\S$ are ineffable. 


Given an ordinal $\beta<\kappa^+$, let us use the notation \[\Pi^1_{<\beta}(\kappa)=\bigcup_{\xi\in\{-1\}\cup\beta}\Pi^1_\xi(\kappa).\] The next corollary is immediate from Lemma \ref{lemma_indescribability_from_homogeneity} together with a straightforward induction on $\gamma$.

\begin{corollary}\label{corollary:levelofhomogeneity}
  Assume that $\mathcal O$ is ineffable, $\beta<\kappa^+$ is an ordinal, and $I\supseteq\Pi^1_{<\beta}(\kappa)$. Then, \[\mathcal O^\gamma(I)\supseteq\Pi^1_{<(\beta+2\gamma)}(\kappa). \] 
\end{corollary}

We will now review material from \cite[Section 13]{HolyLCOandEE} on coding weak $\kappa$-models $M$ and $M$-ultrafilters $U$ on $\kappa$ as subsets of $V_\kappa$. These definitions are tailored so that any transitive weak $\kappa$-model that can be coded will have to be a superset of $V_\kappa$, with elements $x$ of $V_\kappa$ being coded as ordered pairs of the form $\langle 0,x\rangle$, and we code $\kappa$ by $0$.

\begin{definition}
  We say that $\mathcal M\subseteq V_\kappa$ is a \emph{code for a transitive weak $\kappa$-model} if $\mathcal M\subseteq V_\kappa$ with the following properties:
  \begin{itemize}
    \item $\mathcal M$ is a binary relation on $V_\kappa$, such that $\dom(\mathcal M)=V_\kappa$,
    \item for all $x,y\in V_\kappa$, $\langle 0,x\rangle\mathcal M\langle 0,y\rangle$ if and only if $x\in y$,
    \item for all $x$, $x\,\mathcal M\,0\iff \exists y\in\kappa\ x=\langle 0,y\rangle$,
    \item $\mathcal M$ is well-founded and extensional, and
    \item $\langle V_\kappa,\mathcal M\rangle\models\ZFC^-$.
  \end{itemize}
  Note that the weak $\kappa$-model that is coded here is the model $M$ such that $\langle M,\in\rangle$ is the transitive collapse of $\langle V_\kappa,\mathcal M\rangle$. On the other hand, any transitive weak $\kappa$-model $M\supseteq V_\kappa$ has a code as described above, using a suitable bijection between $M$ and $V_\kappa$. Let $\pi_{\mathcal M}$ denote the transitive collapsing map of $\langle V_\kappa,\mathcal M\rangle$. If $X=\pi_{\mathcal M}(x)$, we say that $x$ \emph{is the code of} $X$ (within $\mathcal M$).
\end{definition}

Using standard arguments (see \cite[Lemma 12.2]{HolyLCOandEE}), it is easy to see that the property that $\mathcal M$ is a code for a transitive weak $\kappa$-model is a $\Delta^1_1$-property over $\langle V_\kappa,\mathcal M\rangle$.  
Note also that we can easily shift between subsets $X$ of $V_\kappa$ in $M$ and their codes within $\mathcal M$ using the fact that for $X\subseteq V_\kappa$ in $M$ and $x\in V_\kappa$, the property $\pi_{\mathcal M}^{-1}(X)=x$ is equivalent to the first-order sentence $\forall y\ \left[\langle 0,y\rangle\mathcal M x\longleftrightarrow y\in X\right]$ in $\langle V_\kappa,\in,\mathcal M,X\rangle$.

Next, we want to define what it means to code an $M$-ultrafilter on $\kappa$, which is easily seen to be a $\Delta^1_1$-property over $\langle V_\kappa,\mathcal M,\mathcal U\rangle$.

\begin{definition}
  Given a code $\mathcal M$ for a transitive weak $\kappa$-model $M$, we say that $\mathcal U\subseteq V_\kappa$ is a \emph{code for an $M$-ultrafilter on $\kappa$} if $\langle V_\kappa,\mathcal M,\mathcal U\rangle$ thinks that $\mathcal U$ is an ultrafilter on $0$ (note that our setup is so that $0$ codes $\kappa$).
\end{definition}

For our desired applications, we will need our operators to satisfy some properties of simple definability that were introduced in \cite{HolyLCOandEE}.

\begin{definition}\label{definition_simple}
Let $\langle\Psi,\Omega\rangle$ be a pair of formulas, and let $\mathcal O$ be an ideal operator.
\begin{itemize}
   \item $\langle\Psi,\Omega\rangle$ is \emph{simple} in case $\ZFC$ proves the following:
   \begin{enumerate}
     \item[(a)] whenever $M$ is a transitive weak $\kappa$-model, and $U$ is an $M$-ultrafilter on $\kappa$, then $\Psi(M,U)$ translates to a $\Delta^1_1$-property of any pair of codes $\langle\mathcal M,\mathcal U\rangle$ for $\langle M,U\rangle$ over $V_\kappa$, and
     \item[(b)] whenever the property $X\in I^+$ is definable over $V_\kappa$ by a $\Pi^1_\beta$-formula $\varphi(X)$ for some $0<\beta<\kappa$, then $\Omega(U,I)$ translates to a $\Pi^1_\beta$-property of any code $\mathcal U$ of $U$ over $V_\kappa$.
   \end{enumerate}
   \item $\langle\Psi,\Omega\rangle$ is \emph{always simple} in case $\ZFC$ additionally proves that if in (b), the property $X\in I^+$ is first order definable over $V_\kappa$, then $\Omega(U,I)$ translates to a $\Delta^1_1$-property of any code $\mathcal U$ of $U$ over $V_\kappa$.
   \item $\mathcal O$ is \emph{simple} or \emph{always simple} in case $\ZFC$ proves that $\mathcal O$ can be characterized in the form $\mathcal O=\mathfrak O\Psi\Omega$ for some pair of formulas $\langle\Psi,\Omega\rangle$ that is simple or always simple respectively.
  \end{itemize}
\end{definition}

Definition \ref{definition_simple}(a) is immediate if $\Psi$ can be expressed as a first order property of the structure $\langle M,\in,U\rangle$. For example, this is the case when $\Psi(M,U)$ denotes the statement that $U$ is $\kappa$-amenable for $M$.

The property that $U$ is countably complete translates to the following first order statement about $\mathcal{U}$ over $V_\kappa$: for any countable sequence $\langle u_i\mid i<\omega\rangle$ of elements of $\mathcal U$,\footnote{Since $\kappa$ is assumed to be inaccessible (regular and uncountable suffices), these countable sequences are elements of $V_\kappa$.} there is $x$ such that $x\mathcal M u_i$ for every $i<\omega$.

The statement that $M$ is closed under ${<}\kappa$-sequences translates to the following first order statement about $\mathcal M$ over $V_\kappa$: $\forall p\,\exists t\,\forall x\ \left(x\,\mathcal M\,t\ \iff\ x\in p\right)$.

For other examples, see \cite{HolyLCOandEE}.

\medskip

Let us now look at some examples in which Definition \ref{definition_simple}(b) holds.
\begin{itemize}
  \item If $\Omega(U,I)$ denotes the statement that $U\subseteq I^+$, then this translates to the statement that $\forall x\,\mathcal M\,\mathcal U\,\forall X\ [\pi_{\mathcal M}^{-1}(X)=x\to\varphi(X)]$, where $\varphi$ is a formula defining $I^+$ over $V_\kappa$.
  \item If $\Omega(U,I)$ denotes the property that countable intersections from $U$ are in $I^+$, then this translates to the statement that for any countable sequence $\langle u_\beta\mid\beta<\omega\rangle$ of $\mathcal M$-elements of $\mathcal U$, \[\varphi(\{\alpha<\kappa\mid\forall\beta<\omega\ \langle 0,\alpha\rangle\,\mathcal M\,u_\beta\}).\]
  \item If $\Omega(U,I)$ denotes the property that $\Delta U\in I^+$, then this translates to the statement that for any $\kappa$-enumeration $\langle u_\beta\mid\beta<\kappa\rangle$ of the $\mathcal M$-elements of $\mathcal U$, \[\varphi(\{\alpha<\kappa\mid\forall\beta<\alpha\ \langle 0,\alpha\rangle\,\mathcal M\,u_\beta\}).\]
\end{itemize}
If the property $X\in I^+$ is first order definable, observe that we obtain a $\Delta^1_1$-statement in the first two cases above, for we can equivalently rephrase the above to use existential rather than universal second order quantifiers. However this does not work in the third case (see the remarks made in Footnote \ref{footnote:diagonalintersections}). In particular, this means that the Ramsey operator and the strongly Ramsey subset operator are always simple, while (the model version of) the ineffability operator is simple. 

\begin{remark}\label{remark_more_examples}
Further examples of operators that are both ineffable and always simple have been introduced in \cite[Section 12 and Section 13]{HolyLCOandEE}, including in particular the $\mathbf T_\omega^\kappa$-Ramsey subset operator $\mathbf T$, and the $\mathbf{wf}^\kappa_\omega$-Ramsey subset operator~$\mathbf{wf}$. All of our results on ineffable always simple operators that follow will thus apply to these operators as well.
\end{remark}

%

As a first application, we want to show that Lemma \ref{lemma_complexity} can be extended to work for our framework, and we want to generalize it even further by considering ideals other than the indescribability ideals (which are particular instances of the below by Theorem \ref{theorem_expressing_indescribability}). Note that the lemma below does not include the case of applying the ineffability operator to the bounded ideal, which however is already handled as a special case of Lemma \ref{lemma_complexity}.

\begin{lemma}\label{lemma_complexity extended}
  Suppose $\kappa$ is a regular cardinal, $\gamma,\xi<\kappa^+$ are ordinals with $\xi>0$, $I$ is an ideal on $\kappa$ such that $I^+$ is $\Pi^1_\xi$-definable over $V_\kappa$, $I$ is represented by $\langle I_\alpha\mid\alpha<\kappa\rangle$ in $\Ult$,\footnote{When we refer to a definable ideal $I$ in $\Ult$, we mean the version of $I$ that is obtained by applying that definition in $\Ult$. Strictly speaking, we should thus require that this definition $\ZFC$-provably yields an ideal. This will clearly hold in all relevant cases.} and that $\mathcal O=\mathfrak O\Psi\Omega$ is simple. Then, there is a $\Pi^1_{\xi+2\gamma}$-formula $\Theta^\kappa_{\gamma,\xi}(X)$ over $V_\kappa$ and a club subset $C^\kappa_{\gamma,\xi}$ of $\kappa$ such that for all $S\subseteq\kappa$, we have
\[S\in\O^\gamma(I)^+\text{ if and only if }V_\kappa\models\Theta^\kappa_{\gamma,\xi}(S)\]
and for all regular cardinals $\alpha\in C^\kappa_{\gamma,\xi}$, we have
\[S\cap\alpha\in \O^{f^\kappa_\gamma(\alpha)}(I_\alpha)^+\text{ if and only if }V_\alpha\models\Theta^\kappa_{\gamma,\xi}(S)\res^\kappa_\alpha.\]
If $\mathcal O$ is always simple, then the above also holds for $\xi=0$.
\end{lemma}
\begin{proof}
The second statement is handled as usual, namely it is equivalent to the first statement holding in all generic ultrapowers $\Ult$,\footnote{As before, when we refer to the operator $\mathcal O$ in $\Ult$, we mean the operator $\mathfrak O\Psi\Omega$ in the sense of $\Ult$.} obtained by forcing with $P(\kappa)/\NS_\kappa$. Thus it suffices to verify the first statement both in $V$ and in $\Ult$. We do so by induction on $\gamma<\kappa^+$. The case when $\gamma=0$ is immediate from our assumption. The case when $\gamma$ is a limit ordinal is handled as in the proof of Lemma \ref{lemma_complexity}.

Suppose $\gamma=\delta+1$ and the result holds for $\delta$. Then, there is a $\Pi^1_{\xi+2\delta}$-formula $\Theta^\kappa_{\delta,\xi}(X)$ over $V_\kappa$ such that both in $V$ and in $\Ult$, for all $S\subseteq\kappa$, we have
\begin{align}S\in\mathcal O^{\delta}(I)^+\text{ if and only if }V_\kappa\models\Theta^\kappa_{\delta,\xi}(S).\label{equation_gamma_less_than_omega_at_kappa_extended}
\end{align}
We simply define $\Theta^\kappa_{\gamma,\xi}(X)$ to be the $\Pi^1_{\xi+2\gamma}$-formula over $V_\kappa$ which asserts that for every (code $\mathcal M$ for a) transitive weak $\kappa$-model $M$ there is (a code $\mathcal U$ for) an $M$-ultrafilter $U$ on $\kappa$ such that $\Psi(M,U)$ and $\Omega(U,\mathcal O^\delta(I))$ hold. Since $\mathcal O$ is simple, it follows that this formula is as desired, both in $V$ and in $\Ult$. Clearly, if $\mathcal O$ is always simple, this works also in case $\xi=0$.
\end{proof}

\section{Pre-operators}\label{section:preoperators}

Our ideal operators are defined via local instances that are parametrized by certain objects. Given a cardinal $\kappa$, we refer to the collection of all such objects on $\kappa$ as the \emph{object type at $\kappa$} of such an operator $\mathcal O$, and denote this by $\mathcal T(\mathcal O,\kappa)$. The object type $\mathcal T(\mathcal I,\kappa)$ of the ineffability operator at $\kappa$ is the collection of all $\kappa$-lists, the object type $\mathcal T(\mathcal R,\kappa)$ of the Ramsey operator at $\kappa$ is the collection of all regressive functions $c\colon[\kappa]^{<\omega}\to\kappa$, and the object type of our model based operators at $\kappa$ is simply the powerset of $\kappa$.

Each object type $\mathcal T$ at $\kappa$ comes with an associated restriction operator, which, given some $y\in\mathcal T$ and some $\alpha<\kappa$, outputs its natural restriction $y\restr\alpha$. The following definition should not bear any surprises.

\begin{definition} Suppose $\kappa$ is a cardinal and $\alpha<\kappa$.
  \begin{itemize}
    \item If $\mathcal T=\mathcal P(\kappa)$ and $y\in\mathcal T$, then $y\restr\alpha=y\cap\alpha$.
    \item If $\mathcal T$ is the collection of all $\kappa$-lists and $y\in\mathcal T$, then $y\restr\alpha$ is the restriction of $y$ to the domain $\alpha$, i.e.\ the initial segment of length $\alpha$ of the $\kappa$-sequence $y$.
    \item If $\mathcal T$ is the collection of all functions $c\colon[\kappa]^{<\omega}\to 2$ and $y\in\mathcal T$, then $y\restr\alpha$ is the restriction of $y$ to the domain $[\alpha]^{<\omega}$.
  \end{itemize}
\end{definition}

Each ideal operator $\mathcal O$ with local instances has an associated \emph{pre-operator}.

\begin{definition}
  Given an ideal operator $\mathcal O$ together with local instances $\mathcal O^y$ at $\kappa$ for $y\in\mathcal T(\mathcal O,\kappa)$, we define its \emph{associated pre-operator} $\mathcal O_0$ as follows. Given an ideal $I$ on $\kappa$ such that $I^+$ is definable by a $\Pi^1_\xi$-formula over $V_\kappa$ for some $\xi<\kappa^+$, and such that $I$ (in the sense of $\Ult$) is represented by $\langle I_\alpha\mid\alpha<\kappa\rangle$ in $\Ult$,
 \[\mathcal {O}_0(I)^+=\{x\!\subseteq\!\kappa\mid\forall y\in\mathcal T(\mathcal O,\kappa)\,\forall C\!\subseteq\!\kappa\,\textrm{club }\exists\alpha\!\in\!x\ x\cap C\cap\alpha\in\mathcal O^{y\restr\alpha}(I_\alpha)^+\},\] where $\alpha$ is understood to range over regular uncountable cardinals.
\end{definition}

%
%
%

$\I_0$ is the \emph{subtle operator}, and $\R_0$ is the \emph{pre-Ramsey operator}.

\begin{remark}\label{remark_subtle_ideal}
Notice that by Theorem \ref{theorem_baumgartner}, $\I_0([\kappa]^{<\kappa})=\I_0(\Pi^1_\xi(\kappa))$ is the \emph{subtle ideal} on $\kappa$ for any $\xi<\kappa^+$, that is the collection of all subsets of $\kappa$ which are not subtle. $\R_0([\kappa]^{<\kappa})$ is the \emph{pre-Ramsey ideal} on $\kappa$. Since we do not know whether an analogue of Theorem \ref{theorem_baumgartner} holds for the pre-Ramsey operator (see Question \ref{question:ramseylikeineffable} below), we do not know whether $\R_0(\Pi^1_\xi(\kappa))=\R_0([\kappa]^{<\kappa})$ for all (or any) $\xi<\kappa^+$.
\end{remark}

The second author has shown in \cite{HolyLCOandEE} that the subtle and the pre-Ramsey operators are equivalent to their respective model versions (for ideals containing the nonstationary ideal in case of the subtle operator).

\begin{theorem}\cite[Theorem 7.3 and Theorem 9.1]{HolyLCOandEE}
  Whenever $I\supseteq\NS_\kappa$, \[\I_0(I)=(\I_{mod})_0(I),\] and for arbitrary ideals $I$ on $\kappa$, \[\R_0(I)=(\R_{mod})_0(I).\]
\end{theorem}

\section{Generating ideals}\label{section_generating}

In this section, we analyze the interplay between our generalized operators, and ideals of higher indescribability. Such an analysis for the Ramsey operator $\R$ and the ideals $\Pi^1_\xi(\kappa)$ for $\xi<\kappa$ has been performed by the first author in his \cite{MR4206111}. Let us start this section by citing a classic result of Baumgartner, that we will extend afterwards. Given ideals $I$ and $J$ on $\kappa$, $\overline{I\cup J}$ denotes the collection of all sets $X\cup Y$ for which $X\in I$ and $Y\in J$.

\begin{theorem}\label{Baumgartner_basic}\cite[Section 7]{MR0384553}
  For all cardinals $\kappa$ and all $\xi\in\{-1\}\cup\omega$, $\kappa\in\I(\Pi^1_\xi(\kappa))^+$ if and only if 
\begin{enumerate}
\item $\kappa\in\I_0(\Pi^1_\xi(\kappa))^+\cap\Pi^1_{\xi+2}(\kappa)^+$ and
\item the ideal $\overline{\I_0(\Pi^1_\xi(\kappa))\cup\Pi^1_{\xi+2}(\kappa)}$ is nontrivial and equals $\I(\Pi^1_\xi(\kappa))$.
\end{enumerate}
  Moreover, (2) is necessary in the above characterization of $\kappa\in\I(\Pi^1_\xi(\kappa))^+$, for the least $\Pi^1_{\xi+2}$-indescribable cardinal $\kappa$ such that $\kappa\in\I_0(\Pi^1_\xi(\kappa))^+$ is strictly below the least cardinal $\kappa$ for which $\kappa\in\I(\Pi^1_\xi(\kappa))^+$ (if such exists).
\end{theorem}

We will frequently use the following.

\begin{remark}\label{remark_ideal_generated}
Suppose $I_0$, $I_1$ and $J$ are ideals on $\kappa$. In order to prove that $J=\overline{I_0\cup I_1}$, part of what we must show is that $J\supseteq\overline{I_0\cup I_1}$, or in other words $J^+\subseteq\overline{I_0\cup I_1}^+$. Notice that we may obtain a chain of equivalences directly from the definitions involved:
\begin{align*}
J^+\subseteq\overline{I_0\cup I_1}^+ &\iff \overline{I_0\cup I_1}\subseteq J \\
	&\iff I_0\cup I_1\subseteq J\\
	&\iff J^+\subseteq I_0^+\cap I_1^+.
\end{align*}
\end{remark}

In the following, we extend Baumgartner's result to simple ineffable operators, and to ideals of higher indescribability. For readers who are only interested in the operators $\I$ and $\R$, it should be possible to read this section without having read Section~\ref{section_generalized operators} in full detail. In this case, it is only relevant to know that both $\I$ and $\R$ are ineffable and simple, that $\R$ is always simple, and that $\I$ and $\R$ are \emph{monotonic} in the sense that if $\mathcal O\in\{\I,\R\}$, and $I\subseteq J$ are both ideals on a cardinal $\kappa$, then $\O(I)\subseteq\O(J)$. It should then be easy to read the present section, perhaps checking some relevant bits of Section~\ref{section_generalized operators} when needed. Let us remind our readers of the following, which provides a fairly large class of ideals that Theorem \ref{theorem_generating_iterates} applies to. It is immediate from Corollary \ref{corollary:levelofhomogeneity} and from Lemma \ref{lemma_complexity extended}.

\begin{observation}\label{observation:appliesto}
  Assume that $\mathcal O$ is ineffable and simple, and that $\Pi^1_{<\xi}(\kappa)\subseteq I$.
  \begin{itemize}
    \item If $\gamma<\kappa^+$, $0<\xi<\kappa^+$, and $I^+$ is $\Pi^1_\xi$-definable over $V_\kappa$,\footnote{This is the case in particular if $I=\Pi^1_{<\xi}(\kappa)$.} then $\Pi^1_{<(\xi+2\gamma)}(\kappa)\subseteq\mathcal O^\gamma(I)$, and the latter ideal is $\Pi^1_{\xi+2\gamma}$-definable over $V_\kappa$. 
    \item If $\xi=0$, the same holds true if either $\mathcal O$ is always simple, or if $\gamma\ge\omega$, and $I^+$ is $\Pi^1_n$-definable for some $n<\omega$.\footnote{This is the case in particular if $I=[\kappa]^{<\kappa}$.} \hfill{$\Box$}
  \end{itemize}
\end{observation}

The conclusion of Theorem \ref{theorem_generating_iterates} in case $\O=\R$ and $I=\Pi^1_\xi(\kappa)$ for $\xi<\omega$ is due to Feng in \cite[Theorem 4.8]{MR1077260}, and has been extended to $\xi<\kappa$ by the first author in \cite[Corollary 6.2]{MR4206111}.

\begin{theorem}\label{theorem_generating_iterates}
  Assume that $I$ is an ideal on $\kappa$, $\mathcal O=\mathfrak O\Psi\Omega$ is ineffable and simple, $\gamma,\xi<\kappa^+$, $\xi>0$, $\Pi^1_{<\xi}(\kappa)\subseteq I$, $I^+$ is $\Pi^1_\xi$-definable over $V_\kappa$ and $\kappa\in\O^{\gamma+1}(I)^+$. Then, 
\[\O^{\gamma+1}(I)=\overline{\O_0(\O^\gamma(I))\cup\Pi^1_{\xi+2\gamma+1}(\kappa)}.\]
If either $\mathcal O=\mathcal I$, $\mathcal O$ is ineffable and always simple, or $\gamma\ge\omega$ in the above, then the conclusion also holds in case $\xi=0$.
\end{theorem}
\begin{proof}
Let us first treat the cases when $\mathcal O$ is either simple or always simple.
Let $J=\overline{\O_0(\O^\gamma(I))\cup\Pi^1_{\xi+2\gamma+1}(\kappa)}$, and assume that $I$, in the sense of $\Ult$, is represented by $\langle I_\alpha\mid\alpha<\kappa\rangle$ in $\Ult$.

Suppose $S\in J^+$, and for the sake of a contradiction, suppose $S\in\O^{\gamma+1}(I)$. Let $y\subseteq\kappa$ be such that whenever $M$ is a transitive weak $\kappa$-model with $y\in M$ and $U$ is an $M$-ultrafilter on $\kappa$ such that $\Psi(M,U)$ and $\Omega(U,\mathcal O^\gamma(I))$ hold, then $S\not\in U$. Using that $\mathcal O$ is simple, or always simple in case $\xi=0$, and using Lemma \ref{lemma_complexity extended}, this property of $y$ and of $S$ can be expressed by a natural $\Pi^1_{\xi+2\gamma+1}$-formula $\varphi(y,S)$ over $V_\kappa$, and there is a club $C\subseteq\kappa$ such that for all regular $\alpha\in C$, \[V_\alpha\models\varphi(y,S)\res^\kappa_\alpha\] if and only if whenever $M$ is a transitive weak $\alpha$-model with $y\cap\alpha\in M$ and $U$ is an $M$-ultrafilter on $\alpha$ such that $\Psi(M,U)$ and $\Omega(U,\mathcal O^{f^\kappa_\gamma(\alpha)}(I_\alpha))$ hold, then $S\cap\alpha\not\in U$.

Since $V_\kappa\models\varphi(y,S)$, the set
\[D=\{\alpha<\kappa\st V_\alpha\models\varphi(y,S)\res^\kappa_\alpha\}\]
is in the filter $\Pi^1_{\xi+2\gamma+1}(\kappa)^*$. Since $S\notin J$, $S$ is not the union of a set in $\O_0(\O^\gamma(I))$ and a set in $\Pi^1_{\xi+2\gamma+1}(\kappa)$. Since $S=(S\cap C\cap D)\cup (S\setminus (C\cap D))$ and $S\setminus (C\cap D)\in\Pi^1_{\xi+2\gamma+1}(\kappa)$, we see that $S\cap C\cap D\in\O_0(\O^\gamma(I))^+$. Thus, by definition of $\O_0$, there is some ordinal $\alpha\in S\cap C\cap D$ for which there exists a transitive weak $\alpha$-model $M$ with $y\cap\alpha\in M$ and an $M$-ultrafilter $U$ on $\alpha$ such that $\Psi(M,U)$ and $\Omega(U,\O^{f^\kappa_\gamma(\alpha)}(I_\alpha))$ hold, and such that $S\cap C\cap D\cap\alpha\in U$. However, since $\alpha\in C\cap D$ we have $V_\alpha\models\varphi(y,S)\res^\kappa_\alpha$, contradicting the above.

Now suppose $S\in\O^{\gamma+1}(I)^+$. Since $\mathcal O$ is ineffable on $I$, and by our assumption that $\Pi^1_{<\xi}(\kappa)\subseteq I$, this implies that $S\in\Pi^1_{\xi+2\gamma+1}(\kappa)^+$ by Corollary \ref{corollary:levelofhomogeneity}. Let us show that $S\in\O_0(\O^\gamma(I))^+$. Suppose $y\subseteq\kappa$ and fix a club subset $C$ of $\kappa$. By the third item in Observation \ref{observation:regularity}, it follows that $S\cap C\in\O^{\gamma+1}(I)^+$, and thus there is a weak $\kappa$-model $M\supseteq V_\kappa$ with $y\in M$, and an $M$-ultrafilter $U$ on $\kappa$ such that $\Psi(M,U)$ and $\Omega(U,O^\gamma(I))$ hold, and such that $S\cap C\in U$. By our assumptions and using Lemma \ref{lemma_complexity extended}, this property of $S\cap C$, $y$, and the codes $\mathcal M$ and $\mathcal U$ of $M$ and $U$ respectively is expressible by a $\Pi^1_{\xi+2\gamma}$-formula $\varphi(S\cap C,y,\mathcal M,\mathcal U)$ over $V_\kappa$, which additionally states that $\mathcal M$ is a code for a weak $\kappa$-model $M$, and that $\mathcal U$ is a code for an $M$-ultrafilter on $\kappa$. Moreover using Lemma \ref{lemma_complexity extended}, there is a club subset $D$ of $\kappa$ such that for $\alpha\in D$, $\varphi(S\cap C,y,\mathcal M,\mathcal U)\res^\kappa_\alpha$ expresses the corresponding property over $V_\alpha$, namely that $\mathcal M\cap V_\alpha$ is a code for a weak $\alpha$-model $\bar M$, that $\mathcal U\cap V_\alpha$ is a code for an $\bar M$-ultrafilter $\bar U$ on $\alpha$, that $y\cap\alpha\in\bar M$, $\Psi(\bar M,\bar U)$ and $\Omega(\bar U,\O^{f^\kappa_\gamma(\alpha)}(I_\alpha))$ hold, and that $S\cap C\cap\alpha\in\bar U$. Since $S\cap C\cap D$ is $\Pi^1_{\xi+2\gamma+1}$-indescribable, there is some $\alpha\in S\cap C\cap D$ such that $V_\alpha\models\varphi(S\cap C,y,\mathcal M,\mathcal U)\res^\kappa_\alpha$. Thus, $S\in\O_0(\O^\gamma(I))^+$.

When $\O=\I$ (and $\xi=0$), note that the case when $\gamma=0$ is handled by Theorem \ref{Baumgartner_basic}. For $\gamma>1$, note that if $\gamma=1+\bar\gamma$, we have $\I^\gamma(I)=\I_{mod}^{\bar\gamma}(\I(I))$, and that $\I(I)$ has the properties that $\Pi^1_1(\kappa)\subseteq I$ and that $\I(I)^+$ is $\Pi^1_2$-definable over $V_\kappa$. We can now apply the main case of the theorem using the operator $\mathcal O=\I_{mod}$.

Similarly, if $\gamma=\omega+\delta\ge\omega$ (and $\xi=0$), the desired conclusion of the theorem can be rewritten as \[\O^{\delta+1}(\O^\omega(I))=\overline{\O_0(\O^\delta(\O^\omega(I)))\cup\Pi^1_{\xi+2\gamma+1}(\kappa)},\]
and we can deduce this conclusion from applying the theorem to the ideal $\O^\omega(I)$, using Observation \ref{observation:appliesto}.
\end{proof}

\subsection{On finite iterates of operators}

By Remark \ref{remark_subtle_ideal}, we have $\I_0(\Pi^1_\xi(\kappa))=\I_0([\kappa]^{<\kappa})$ for any $\xi<\kappa^+$, and hence we easily obtain the following corollary of Theorem \ref{theorem_generating_iterates}.

\begin{corollary}\label{corollary_ineffabledowntozero}
Suppose $\kappa\in\I(\Pi^1_\xi(\kappa))^+$ where $\xi\in\{-1\}\cup\kappa^+$. Then
\[\I(\Pi^1_\xi(\kappa))=\overline{\I_0([\kappa]^{<\kappa})\cup\Pi^1_{\xi+2}(\kappa)}.\]
\end{corollary}

We can obtain a variant of Theorem \ref{theorem_generating_iterates} for finite iterates of operators as follows.

\begin{corollary}\label{corollary_generating_finite_iterates}
  Assume that $I$ is an ideal on $\kappa$, $\mathcal O$ is ineffable and simple, $\gamma<\omega$, $0<\xi<\kappa^+$, $\Pi^1_{<\xi}(\kappa)\subseteq I$, and $I^+$ is $\Pi^1_\xi$-definable over $V_\kappa$. Then,
\[\O^{\gamma+1}(I)=\overline{\O_0(\O^\gamma(I))\cup\O^\gamma(\Pi^1_{\xi+1}(\kappa))}.\]
If either $\mathcal O=\mathcal I$, $\mathcal O$ is ineffable and always simple, or $\gamma\ge\omega$ in the above, then the above conclusion also holds in case $\xi=0$.
\end{corollary}
\begin{proof}
  On the one hand, by our assumptions and Corollary \ref{corollary:levelofhomogeneity}, $\O(I)\supseteq\Pi^1_{\xi+1}(\kappa)$, and therefore, $\O^{\gamma+1}(I)\supseteq\O^\gamma(\Pi^1_{\xi+1}(\kappa))$ by the monotonicity of $\O$ (see Proposition \ref{proposition:abstractbasic}). On the other hand, $\O^\gamma(\Pi^1_{\xi+1}(\kappa))\supseteq\Pi^1_{\xi+2\gamma+1}(\kappa)$. Thus, the result follows immediately from Theorem \ref{theorem_generating_iterates}.
\end{proof}

There is also a sort of analogue of the above for infinite iterates of operators. This has been worked out for the Ramsey operator in \cite[Theorem 7.8]{MR4206111}, and can analogously be performed for our generalized operators. We will leave all details to the interested reader.

As an easy corollary of Corollary \ref{corollary:levelofhomogeneity}, again using the monotonicity of our operators, we obtain the following generalization of \cite[Corollary 6.8]{MR4206111}:

\begin{corollary}\label{corollary_collapse}
If $\O$ is ineffable, $\xi\in\{-1\}\cup\kappa^+$ and $n<\omega$, then
\[\O^\omega(\Pi^1_\xi(\kappa))=\O^\omega(\Pi^1_{\xi+n}(\kappa)).\]
\end{corollary}

The next corollary is a starting point in relating assumptions of the form $\kappa\in\O^\gamma(\Pi^1_\xi(\kappa))$, for different $\gamma$ and $\xi$ below $\kappa^+$, with respect to consistency strength.

\begin{corollary}\label{corollary_xi_to_xi_plus_one_hierarchy}
Assume that $\mathcal O$ is ineffable and simple. Suppose $\gamma<\omega$, $\xi<\kappa^+$, and $\kappa\in\O^\gamma(\Pi^1_\xi(\kappa))^+$. If $S\in\O^\delta(\Pi^1_{\zeta}(\kappa))^+$ where $\zeta+1+2\delta\leq\xi+2\gamma$, then
\[T=\{\alpha<\kappa\textrm{ regular}\st S\cap\alpha\in\O^{f^\kappa_\delta(\alpha)}(\Pi^1_{f^\kappa_{\zeta}(\alpha)}(\alpha))^+\}\in\O^\gamma(\Pi^1_\xi(\kappa))^*.\]
If either $\mathcal O=\mathcal I$, or $\mathcal O$ is ineffable and always simple, then the above conclusion also holds in case $\xi=-1$.
\end{corollary}
\begin{proof}
The fact that $S\in\O^\delta(\Pi^1_\zeta(\kappa))^+$ is expressible by a $\Pi^1_{\zeta+1+2\delta}$-formula $\Theta$ over $V_\kappa$ by Lemma \ref{lemma_complexity extended}, or by Lemma \ref{lemma_complexity} in case $\O=\I$. Let $C$ be the corresponding club obtained from the relevant lemma. Since $\zeta+1+2\delta\leq\xi+2\gamma$, we have $\Pi^1_{\zeta+1+2\delta}(\kappa)\subseteq\Pi^1_{\xi+2\gamma}(\kappa)\subseteq\O^\gamma(\Pi^1_\xi(\kappa))$ by Corollary \ref{corollary:levelofhomogeneity}. The set
\[\{\alpha\in C\textrm{ regular}\st V_\alpha\models\Theta(S)\res^\kappa_\alpha\}\]
is contained in $T$ and is in $\Pi^1_{\zeta+1+2\delta}(\kappa)^*\subseteq\O^\gamma(\Pi^1_\xi(\kappa))^*$. Therefore, $T\in\O^\gamma(\Pi^1_\xi(\kappa))^*$, as desired.
\end{proof}

%

We do not know whether the next result on the proper containment of certain ideals generated by applications of $\I$ and $\R$ generalizes to our framework of operators, for we do not know whether Lemma \ref{lemma_set_of_nons_is_positive} does.

\begin{corollary}\label{corollary_proper_containment_from_xi_to_xi_plus_one}
Let $\O\in\{\I,\R\}$. Suppose $\gamma<\omega$, $\xi\in\{-1\}\cup\kappa^+$ and $\kappa\in\O^\gamma(\Pi^1_{\xi+1}(\kappa))^+$. Then, \[\O^\gamma(\Pi^1_\xi(\kappa))\subsetneq\O^\gamma(\Pi^1_{\xi+1}(\kappa)).\]
\end{corollary}

\begin{proof}
Clearly $\O^\gamma(\Pi^1_\xi(\kappa))\subseteq\O^\gamma(\Pi^1_{\xi+1}(\kappa))$, so we just need to show that the containment is proper. Let $S=\{\alpha<\kappa\st\alpha\in\O^{f^\kappa_\gamma(\alpha)}(\Pi^1_{f^\kappa_\xi(\alpha)}(\alpha))\}$. Then $S\in\O^\gamma(\Pi^1_\xi(\kappa))^+$ by Lemma \ref{lemma_set_of_nons_is_positive}, and Corollary \ref{corollary_xi_to_xi_plus_one_hierarchy} implies that $S\in\O^\gamma(\Pi^1_{\xi+1}(\kappa))$.
\end{proof}

We can show yet another form of proper containment of ideals when $\O=\I$. An analogous result for the operator $\R$ was claimed by the first author in \cite{MR4206111}, see our Question \ref{question:propercontainment} below.

\begin{corollary}\label{corollary_proper_containment_from_baumgartner}
Suppose $\gamma<\kappa^+$, $\xi\in\{-1\}\cup\kappa^+$ and $\kappa\in\I^\gamma(\Pi^1_{\xi+2}(\kappa))^+$. Then, \[\I^\gamma(\Pi^1_{\xi+2}(\kappa))\subsetneq\I^{\gamma+1}(\Pi^1_\xi(\kappa)).\footnote{Let us remark that by Corollary \ref{corollary_collapse}, the case when $\gamma\ge\omega$ is in fact trivial, and the result would hold for arbitrary ineffable operators from our framework in this case.}\] 
\end{corollary}

\begin{proof}
The inclusion itself is immediate, since $\Pi^1_{\xi+2}(\kappa)\subseteq\I(\Pi^1_\xi(\kappa))$ by the ineffability of $\I$, and it only remains to verify its properness. Since $\kappa\in\I^\gamma(\Pi^1_{\xi+2}(\kappa))^+$, it follows by Lemma \ref{lemma_set_of_nons_is_positive} that the set \[S=\{\alpha<\kappa\st\alpha\in\I^{f^\kappa_\gamma(\alpha)}(\Pi^1_{f^\kappa_{\xi+2}(\alpha)}(\alpha))\}\] is in $\I^\gamma(\Pi^1_{\xi+2}(\kappa))^+$. From Corollary \ref{corollary_below_gamma_plus_1_almost_ineffability}, it follows that the set
\[C=\{\alpha<\kappa\st(\forall\eta<\alpha^+)\ \alpha\in\I^{f^\kappa_\gamma(\alpha)}(\Pi^1_\eta(\alpha))^+\}\]
is in the filter $\I^{\gamma+1}([\kappa]^{<\kappa})^*$. Since $C\subseteq\kappa\setminus S$, we see that $\kappa\setminus S\in \I^{\gamma+1}([\kappa]^{<\kappa})^*\subseteq\I^{\gamma+1}(\Pi^1_\xi(\kappa))^*$. Hence, this implies that $S\in\I^{\gamma+1}(\Pi^1_\xi(\kappa))\setminus\I^\gamma(\Pi^1_{\xi+2}(\kappa))$.
\end{proof}


The next corollary extends Baumgartner's observation that the use of ideals is necessary in Theorem \ref{Baumgartner_basic}.

\begin{corollary}\label{corollary_characterization}
  Assume that $\mathcal O=\mathfrak O\Psi\Omega$ is ineffable and simple, $\gamma<\omega$, $\xi<\kappa$, and $I=\Pi^1_{<\xi}(\kappa)$. Then, $\kappa\in\O^{\gamma+1}(I)^+$ if and only if 
\begin{enumerate}
\item $\kappa\in \O_0(\O^\gamma(I))^+\cap\Pi^1_{\xi+2\gamma+1}(\kappa)^+$ and
\item the ideal $\overline{\O_0(\O^\gamma(I))\cup\Pi^1_{\xi+2\gamma+1}(\kappa)}$ is nontrivial and equals $\O^{\gamma+1}(I)$.
\end{enumerate}
If either $\mathcal O=\mathcal I$, or $\mathcal O$ is ineffable and always simple in the above, then the conclusion also holds in case $\xi=0$.

Moreover, (2) is necessary in the the above characterization, that is, the least $\Pi^1_{\xi+2\gamma+1}$-indescribable cardinal $\kappa$ that satisfies $\kappa\in\O_0(\O^\gamma(I))^+$ is strictly below the least cardinal $\kappa$ that satisfies $\kappa\in\O^{\gamma+1}(I)^+$.
\end{corollary}
\begin{proof}
  Note that $\kappa$ being $\Pi^1_{\xi+2\gamma+1}$-indescribable and $\kappa\in\O_0(\O^\gamma(I))^+$ are $\Pi^1_{\xi+2\gamma+2}$-properties over $V_\kappa$, and $\kappa\in\O^{\gamma+1}(I)^+$ implies that $\kappa$ is $\Pi^1_{<(\xi+1+2\gamma+2)}$-indescribable by Corollary \ref{corollary:levelofhomogeneity}, and hence $\kappa$ is $\Pi^1_{\xi+2\gamma+2}$-indescribable using that $\gamma$ is finite. Now since $\xi<\kappa$, this yields some $\xi<\alpha<\kappa$ such that $\alpha\in\O_0(\O^\gamma(\Pi^1_\xi(\alpha)))^+$ and $\alpha\in\Pi^1_{\xi+2\gamma+1}(\alpha)^+$.
\end{proof}

In the above, one could obtain analogous results when $\kappa\le\xi<\kappa^+$, however the statement that is reflected down from $\kappa$ to $\alpha$ will be changed for $\xi$ will be reflected down to $f^\kappa_\xi(\alpha)$. This still yields a satisfactory analogue of Corollary \ref{corollary_characterization} when $\kappa\leq\xi<\kappa^+$ and $\xi$ is definable from $\kappa$ (for example, if $\xi=\kappa$, or $\xi=\kappa+\kappa$ etc.). We will leave the easy and straightforward details to our interested readers.

\subsection{On infinite iterates of operators}

We would like to use Theorem \ref{theorem_generating_iterates} to prove an analogue of Corollary \ref{corollary_xi_to_xi_plus_one_hierarchy} for infinite $\gamma$, which would, in a sense, say that the strength of the hypothesis ``$\kappa\in\O^\gamma(\Pi^1_\xi(\kappa))^+$'' increases as $\xi$ increases. However, there is an added complication, as illustrated in Corollary \ref{corollary_collapse}, which is that if $\xi_0<\xi_1<\kappa^+$, it may be that $\kappa\in\O^\gamma(\Pi^1_{\xi_0}(\kappa))^+$ is equivalent to $\kappa\in\O^\gamma(\Pi^1_{\xi_1}(\kappa))^+$, if $\gamma$ is large enough. In the next theorem, we determine the least $\gamma$ for which this occurs when $\O\in\{\I,\R\}$. Let us note that we do not know how to verify this leastness for operators other than $\I$ and $\R$. Even though the other statements of the theorem below in fact hold for simple ineffable operators, we therefore only state the below result for these two operators.

\begin{theorem}
Suppose $\kappa$ is a cardinal, $\xi_0<\xi_1$ are in $\{-1\}\cup\kappa^+$ and $\O\in\{\I,\R\}$.
Then, the ideal chains \[\<\O^\gamma(\Pi^1_{\xi_0}(\kappa))\st\gamma<\kappa^+\>\textrm{ and }\<\O^\gamma(\Pi^1_{\xi_1}(\kappa))\st\gamma<\kappa^+\>\] are eventually equal. Moreover, letting $\delta=\ot(\xi_1\setminus\xi_0)\cdot\omega$, if the ideal $\O^\delta(\Pi^1_{\xi_1}(\kappa))$ is nontrivial, then $\delta$ is least ordinal such that \[\O^\delta(\Pi^1_{\xi_0}(\kappa))=\O^\delta(\Pi^1_{\xi_1}(\kappa)).\]
\end{theorem}

\begin{proof}
First, let us show that $\O^\delta(\Pi^1_{\xi_0}(\kappa))=\O^\delta(\Pi^1_{\xi_1}(\kappa))$, where $\delta=\ot(\xi_1\setminus\xi_0)\cdot\omega$. Since $\xi_0<\xi_1$, it is clear that $\O^\delta(\Pi^1_{\xi_0}(\kappa))\subseteq\O^\delta(\Pi^1_{\xi_1}(\kappa))$. Let us show that $\O^\delta(\Pi^1_{\xi_0}(\kappa))\supseteq\O^\delta(\Pi^1_{\xi_1}(\kappa))$. If $\sigma=\ot(\xi_1\setminus\xi_0)=n$ is finite, then $\delta=n\cdot\omega=\omega$ and the result follows from Corollary \ref{corollary_collapse}. Suppose $\sigma\geq\omega$. Then $\delta=\sigma\cdot\omega$ is a limit of limit ordinals. Thus, it will suffice to show that $\O^\eta(\Pi^1_{\xi_1}(\kappa))\subseteq\O^\delta(\Pi^1_{\xi_0}(\kappa))$ for all limit ordinals $\eta<\delta$. Fix a limit ordinal $\eta<\delta$. By Corollary \ref{corollary:levelofhomogeneity}, we have
\begin{align}
\Pi^1_{\xi_1}(\kappa)=\Pi^1_{\xi_0+\sigma}(\kappa)\subseteq\O^{\sigma+1}(\Pi^1_{\xi_0}(\kappa)).\label{equation_least}
\end{align}
Applying the operator $\O$ $\eta$-many times to (\ref{equation_least}) yields
\[\O^\eta(\Pi^1_{\xi_1}(\kappa))\subseteq\O^{\sigma+1+\eta}(\Pi^1_{\xi_0}(\kappa))\subseteq\O^{\delta}(\Pi^1_{\xi_0}(\kappa)),\]
where the final subset relation follows since $\sigma+1+\eta<\delta$.


Next, let us show that if $\eta<\delta$, then $\O^\eta(\Pi^1_{\xi_0}(\kappa))\subsetneq\O^\eta(\Pi^1_{\xi_1}(\kappa))$. If $\sigma=\ot(\xi_1\setminus\xi_0)$ is finite, in which case $\delta=\omega$, then the result follows from Corollary \ref{corollary_proper_containment_from_xi_to_xi_plus_one}. On the other hand, if $\sigma$ is infinite, then $\delta=\sigma\cdot\omega$ is a limit of limit ordinals. Let $\nu$ be a limit ordinal with $\eta\le\nu<\delta$. It suffices to show that $\O^{\nu+1}(\Pi^1_{\xi_0}(\kappa))\subsetneq\O^{\nu+1}(\Pi^1_{\xi_1}(\kappa))$, for this contradicts $\O^\eta(\Pi^1_{\xi_0}(\kappa))=\O^\eta(\Pi^1_{\xi_1}(\kappa))$. Let
\[S=\left\{\alpha<\kappa\st\alpha\in\O^{f^\kappa_{\nu+1}(\alpha)}(\Pi^1_{f^\kappa_{\xi_0}(\alpha)}(\alpha))\right\}.\]
Since $\kappa\in\O^\delta(\Pi^1_{\xi_1}(\kappa))^+$, it follows from Lemma \ref{lemma_set_of_nons_is_positive} that $S\notin\O^{\nu+1}(\Pi^1_{\xi_0}(\kappa))$. Furthermore, the fact that $S\notin\O^{\nu+1}(\Pi^1_{\xi_0}(\kappa))$ is expressible by a $\Pi^1_{\xi_0+\nu+2}$-sentence $\Theta$ over $V_\kappa$, by Lemma \ref{lemma_complexity}. Let $C$ be the corresponding club subset of $\kappa$ obtained from that lemma. It follows that the set
\[D=\{\alpha\in C\st V_\alpha\models\Theta(S)\res^\kappa_\alpha\}\]
is in the filter $\Pi^1_{\xi_0+\nu+2}(\kappa)^*$ and is contained in $\kappa\setminus S$. Hence, $S\in\Pi^1_{\xi_0+\nu+2}(\kappa)$. By Corollary \ref{corollary:levelofhomogeneity}, since $\xi_1=\xi_0+\sigma$, it follows that $\Pi^1_{\xi_0+\sigma+\nu+1}(\kappa)\subseteq\O^{\nu+1}(\Pi^1_{\xi_1}(\kappa))$. Since $\nu<\delta=\sigma\cdot\omega$, it follows that $\xi_0+\nu+2<\xi_0+\sigma+\nu+1$ and thus $\Pi^1_{\xi_0+\nu+2}(\kappa)\subseteq\O^{\nu+1}(\Pi^1_{\xi_1}(\kappa))$. Together with the above, this implies that $S\in\O^{\nu+1}(\Pi^1_{\xi_1}(\kappa))$.
\end{proof}



Next, extending Corollary \ref{corollary_xi_to_xi_plus_one_hierarchy} to infinite iterates of operators, we show that for ineffable and simple operators $\O$, $\gamma<\kappa^+$ and $\xi_0<\xi_1$ in $\kappa^+\setminus\omega$, the hypothesis $\kappa\in\O^\gamma(\Pi^1_{\xi_1}(\kappa))^+$ implies that there are many $\alpha<\kappa$ which satisfy $\alpha\in\O^{f^\kappa_\gamma(\alpha)}(\Pi^1_{f^\kappa_{\xi_0}(\alpha)}(\alpha))^+$, \emph{assuming $\xi_0$ and $\xi_1$ are far enough apart}. Thus, the hypotheses of the form $\kappa\in\O^\gamma(\Pi^1_\xi(\kappa))^+$ for (certain) $\xi<\kappa^+$ provide a strictly increasing hierarchy of length $\kappa^+$.

\begin{theorem}
Suppose $\O$ is ineffable and simple, $\kappa$ is a cardinal, $\omega\le\xi_0<\xi_1$ are in $\{-1\}\cup\kappa^+$ and $\gamma<\ot(\xi_1\setminus\xi_0)\cdot\omega$. If $\kappa\in\O^\gamma(\Pi^1_{\xi_1}(\kappa))^+$, then the set

\[\{\alpha<\kappa\st\alpha\in\O^{f^\kappa_\gamma(\alpha)}(\Pi^1_{f^\kappa_{\xi_0}(\alpha)}(\alpha))^+\}\]
is in $\O^\gamma(\Pi^1_{\xi_1}(\kappa))^*$.
\end{theorem}
\begin{proof}
 Since $\kappa\in\O^\gamma(\Pi^1_{\xi_1}(\kappa))^+$ and $\xi_0<\xi_1$, we have $\kappa\in\O^\gamma(\Pi^1_{\xi_0}(\kappa))^+$, which is expressible by a $\Pi^1_{\xi_0+1+2\gamma}$-formula $\Theta$ over $V_\kappa$ by Lemma \ref{lemma_complexity extended}. Let $C$ be the corresponding club subset of $\kappa$ obtained from that lemma. Since $\gamma<\ot(\xi_1\setminus\xi_0)\cdot\omega$, it follows that $\xi_0+1+2\gamma<\xi_1+1+2\gamma$. Now by Corollary \ref{corollary:levelofhomogeneity}, we see that $\Pi^1_{<(\xi_1+1+2\gamma)}\subseteq\O^\gamma(\Pi^1_{\xi_1}(\kappa))$, and thus, the set
\[D=\{\alpha\in C\st V_\alpha\models\Theta(\kappa)\res^\kappa_\alpha\}\subseteq\{\alpha<\kappa\st\alpha\in\O^{f^\kappa_\gamma(\alpha)}(\Pi^1_{f^\kappa_{\xi_0}(\alpha)}(\alpha))^+\}\]
is in $\O^\gamma(\Pi^1_{\xi_1}(\kappa))^*$.
\end{proof}


\section{On some results of the first author}\label{Cody_results}

In \cite[Theorem 4.1]{MR4206111}, the first author claimed the following: If $S\in\mathcal R([\kappa]^{<\kappa})^+$, then \[T=\{\alpha<\kappa\mid\forall\xi<\alpha\  S\cap\alpha\in\Pi^1_\xi(\alpha)^+\}\in\mathcal R([\kappa]^{<\kappa})^*.\]
 The proof that is provided however is slightly flawed, and in fact only yields a somewhat weaker result, namely a weak form of the analogue of Theorem \ref{theorem_pushing_Baumgartner} for the Ramsey operator $\mathcal R$ rather than the ineffability operator $\mathcal I$, in the special case when $\gamma=0$ (see below). We first want to provide a counterexample for the above statement that is claimed in \cite{MR4206111}, and then follow it with a corrected version of that theorem. We then shortly discuss the consequences that this has on other results of \cite{MR4206111}. For the very start, we need an auxiliary result.
 
 \begin{lemma}
  If $\kappa$ is a measurable cardinal, then 
  \[\{\alpha<\kappa\mid\alpha\textrm{ is not Ramsey}\}\not\in\mathcal R([\kappa]^{<\kappa})^*.\] 
\end{lemma}
\begin{proof}
  Using Theorem \ref{theorem:ramseymodels}, $\mathcal R([\kappa]^{<\kappa})^*=\R_{mod}([\kappa]^{<\kappa})^*$. Let $U^*$ be a measurable ultrafilter on $\kappa$ and let $A\subseteq\kappa$ be arbitrary. Let $M^*\prec H((2^\kappa)^+)$ have size $\kappa$ with $A,U^*\in M^*$ and such that $\kappa+1\subseteq M^*$, let $M$ be the transitive collapse of $M^*$, and let $U$ be the image of $U^*$ under the collapsing map. Then, $U$ is $M$-normal, $\kappa$-amenable for $M$ and countably complete, and since $\kappa$ is Ramsey in the ultrapower of $V$ by $U^*$, it is also Ramsey in the ultrapower of $M$ by $U$, and hence $\{\alpha<\kappa\mid\alpha$ is Ramsey$\}\in U$. This shows that $\{\alpha<\kappa\mid\alpha$ is not Ramsey$\}\not\in\R_{mod}([\kappa]^{<\kappa})^*$.
\end{proof}

\begin{counterexample}
  Assume that $\kappa$ is Ramsey, such that \[S=\{\alpha<\kappa\mid\alpha\textrm{ is not Ramsey}\}\not\in\mathcal R([\kappa]^{<\kappa})^*.\] 
  Then,
  \[T=\{\alpha<\kappa\mid\forall\xi<\alpha\ S\cap\alpha\in\Pi^1_\xi(\alpha)^+\}\subseteq
  \{\alpha<\kappa\mid S\cap\alpha\in\Pi^1_2(\alpha)^+\}=\]
  \[=\{\alpha<\kappa\mid\{\beta<\alpha\mid\beta\textrm{ is not Ramsey}\}\in\Pi^1_2(\alpha)^+\}.\]
  But since being Ramsey is a $\Pi^1_2$-property, if $\alpha$ is a Ramsey cardinal, then every set in $\Pi^1_2(\alpha)^+$ contains a Ramsey cardinal. Hence the latter set, and thus also $T$, is contained in $S$. This shows that $T\not\in\mathcal R([\kappa]^{<\kappa})^*$.
\end{counterexample}

The following seems to be exactly the statement that is shown to hold true by the proof of \cite[Theorem 4.1]{MR4206111}.

\begin{theorem}\label{theorem:codycorrect}
  If $\kappa$ is a cardinal, $S\in\mathcal R([\kappa]^{<\kappa})^+$, and \[T=\{\alpha<\kappa\mid\forall\xi<\alpha\ S\cap\alpha\in\Pi^1_\xi(\alpha)^+\},\] then $S\setminus T\in\mathcal R([\kappa]^{<\kappa})$.
\end{theorem}

As for the case when $\gamma=0$ in the proof of Theorem \ref{theorem_pushing_Baumgartner} however, this result now follows directly from Theorem \ref{theorem_baumgartner} because the subtle ideal is contained in the Ramsey ideal $\R([\kappa]^{<\kappa})$. Note that Theorem \ref{theorem_baumgartner} in fact yields the stronger statement that if \[T^*=\{\alpha<\kappa\mid\forall\xi<\alpha^+\ S\cap\alpha\in\Pi^1_\xi(\alpha)^+\},\] then $S\setminus T^*\in\mathcal R([\kappa]^{<\kappa})$.

The next result that is claimed in \cite{MR4206111} is its \cite[Theorem 4.2]{MR4206111}, which suffers the same kind of problem as does its \cite[Theorem 4.1]{MR4206111} (and it is now seen to be wrong for it includes the base case when $\alpha=0$, which is \cite[Theorem 4.1]{MR4206111}). However, if its statement is modified according to the modification of \cite[Theorem 4.1]{MR4206111} that we provided in Theorem \ref{theorem:codycorrect} above, it is not clear as to whether Cody's argument can be adapted to work. Let us thus state what might be a good candidate for a corrected version of \cite[Theorem 4.2]{MR4206111} as an open question.\footnote{In fact, we ask a strong version of this question, allowing for $\xi<\alpha^+$ rather than just $\xi<\alpha$.}

\begin{question}\label{question:ramseylikeineffable2}
  Assume that $\kappa$ is a cardinal, that $\gamma<\kappa$ is an ordinal, that $S\in\mathcal R^{\gamma+1}([\kappa]^{<\kappa})^+$, and that $T=\{\alpha<\kappa\mid\forall\xi<\alpha^+\ S\cap\alpha\in\mathcal R^\gamma(\Pi^1_\xi(\alpha))^+\}$. Does it follow that $S\setminus T\in\mathcal R^{\gamma+1}([\kappa]^{<\kappa})$?\footnote{If one tries to adapt the proof of \cite[Theorem 4.2]{MR4206111} in a seemingly obvious way, making use of the notation from that proof, one defines sets $X$ and $H$ in $\mathcal R^{\alpha_0+1}([\kappa]^{<\kappa})^+$, and $C$, as in Cody's argument, but then, the inductive conclusion is that $H\setminus C$ is in $\mathcal R^{\alpha_0+1}([\kappa]^{<\kappa})$, rather than Cody's inductive conclusion that $C\in\mathcal R^{\alpha_0+1}([\kappa]^{<\kappa})^*$. Now our weaker conclusion doesn't seem to allow us to derive that $X\cap C\in\mathcal R^{\alpha_0+1}([\kappa]^{<\kappa})^+$.}
\end{question}

\cite[Theorem 4.2]{MR4206111} is then used to deduce \cite[Corollary 4.3]{MR4206111}, which we would like to pose as yet another open question, since it now seems unclear how to prove the below when $\gamma>0$ (its instance for $\gamma=0$ however follows directly from Theorem \ref{theorem:codycorrect}).

\begin{question}\label{question:pushinguptheramseyhierarchy}
  Assume that $\kappa\in\mathcal R^{\gamma+1}([\kappa]^{<\kappa})^+$. Does it follow that \[\{\alpha<\kappa\mid\forall\xi<\alpha\ \alpha\in\mathcal R^\gamma(\Pi^1_\xi(\alpha))^+\}\in\mathcal R^{\gamma+1}([\kappa]^{<\kappa})^*?\]
\end{question}

In the remainder of \cite{MR4206111}, the results from its Section 4 are only used in a few places. The first result that becomes unclear is \cite[Theorem 6.7 (2)]{MR4206111} (except for the case when $m=1$, for the proof of which the case when $\gamma=0$ in Question \ref{question:pushinguptheramseyhierarchy} suffices), which we thus state as an open question.

\begin{question}\label{question:propercontainment}
  Suppose $1<m<\omega$ and $\xi<\kappa$.
  If $\kappa\in\R^m(\Pi^1_\xi(\kappa))^+$, does it follow that the inclusion \[\R^{m-1}(\Pi^1_{\xi+2}(\kappa))\subseteq\R^m(\Pi^1_\xi(\kappa))\] is a proper inclusion?
\end{question}

The only other result from \cite{MR4206111} that becomes unclear is (the properness of the containments in) \cite[Theorem 7.9]{MR4206111}, which is essentially a version of \cite[Theorem 6.7 (2)]{MR4206111} (and thus would yield a version of Question \ref{question:propercontainment}) for infinite $m$.

In order to answer Question \ref{question:ramseylikeineffable2} in the affirmative, it seems one would need to proceed by induction on $\gamma$ to prove a statement similar to that of Theorem \ref{theorem_pushing_Baumgartner}, but with the ineffable operator $\I$ replaced with the Ramsey operator $\R$ and the $S$-list $\vec{S}$ replaced with a regressive function. This suggests the following.

\begin{question}\label{question_Baumgartner_for_Ramsey_hierarchy}
Suppose $\gamma<\kappa^+$, $S\in \R^{\gamma+1}([\kappa]^{<\kappa})^+$ and $f:[S]^{<\omega}\to \kappa$ is a regressive function. Let $A$ be the set of all ordinals $\alpha\in S$ such that
\[\exists X\subseteq S\cap\alpha\left[(\forall \xi<\alpha^+\ X\in \R^{f^\kappa_\gamma(\alpha)}(\Pi^1_\xi(\alpha))^+) \land (X\cup\{\alpha\}\text{ is hom. for }\vec{S})\right].\]
Does it follow that $S\setminus A\in\R^{\gamma+1}([\kappa]^{<\kappa})$?
\end{question}
Notice that in order to address Question \ref{question_Baumgartner_for_Ramsey_hierarchy}, one might attempt an argument similar to that of Theorem \ref{theorem_baumgartners_lemma_for_the_strongly_ramsey_ideal}, using Ramseyness embeddings instead of strong Ramseyness embeddings. However, the elementary embedding characterization of Ramseyness involves weak $\kappa$-models which are not in general closed under $\omega$-sequences, and therefore, in the context of the proof of Theorem \ref{theorem_baumgartners_lemma_for_the_strongly_ramsey_ideal}, if one only assumes that $M$ is a weak $\kappa$-model, there is no reason to expect that the sequence $\<B_n\st n<\omega\>$ is in $M$, and hence $B=\bigcap_{n<\omega}B_n$ may not be in $M$.

Feng \cite{MR1077260} showed that the Ramsey operator can be characterized using $(\omega,S)$\--sequences. Recall that for any set $S$ of ordinals, an \emph{$(\omega,S)$-sequence} is a sequence $\vec S$ of the form $\vec S=\langle S_{\alpha_1\ldots\alpha_n}\mid 1\le n<\omega,\,\alpha_1<\ldots<\alpha_n,\,\alpha_1,\ldots,\alpha_n\in S\rangle$, where each $S_{\alpha_1\ldots\alpha_n}\subseteq\alpha_1$. We say that $H\subseteq S$ is \emph{homogeneous} for $\vec S$ if for all $n>0$, and all $\alpha_1<\ldots<\alpha_n$ and $\beta_1<\ldots<\beta_n$ from $H$, if $\alpha_1\le\beta_1$, then $S_{\alpha_1\ldots\alpha_n}=S_{\beta_1\ldots\beta_n}\cap\alpha_1$. Feng proved that for any ideal $I$ on a regular cardinal $\kappa$ we have $S\in\R(I)^+$ if and only if every $(\omega,S)$-list has a homogeneous set $H\in P(S)\cap I^+$. Thus, in Question \ref{question_Baumgartner_for_Ramsey_hierarchy} one may replace the regressive function with an $(\omega,S)$-list if desired.

It seems that in order to handle the base case ($\gamma=0$) of Question \ref{question_Baumgartner_for_Ramsey_hierarchy}, one would want to address the following question about the pre-Ramsey ideal; this is in analogy to the fact that the base case of Theorem \ref{theorem_pushing_Baumgartner} follows from the corresponding result, namely Theorem \ref{theorem_generalizing_Baumgartner}, about the subtle ideal. It is straightforward to check that the pre-Ramsey ideal can be characterized in terms of $(\omega,S)$-sequences, so let us formulate the question as follows.

\begin{question}\label{question:ramseylikeineffable}
  If $S\in\mathcal R_0([\kappa]^{<\kappa})^+$, $\vec a$ is an $(\omega,S)$-sequence, and \[A=\{\alpha\!\in\!S\mid\exists X\!\subseteq\!S\cap\alpha\,\forall\xi<\alpha^+\ X\!\in\!\Pi^1_\xi(\alpha)^+\,\land\,X\!\cup\!\{\alpha\}\textrm{ is homogeneous for }\vec a\},\] does it follow that $S\setminus A\in\mathcal R_0([\kappa]^{<\kappa})$?
\end{question}

Furthermore, note that in Theorem \ref{theorem_generalizing_Baumgartner}, we are only stating a particular instance of Baumgartner's original result, for it is not only about subtlety, but about \emph{$n$-subtlety} for any particular $n<\omega$; here $n$-subtley is a property which resembles subtlety but is formulated in terms of $(n,S)$-sequences (see \cite{MR0384553}). Since pre-Ramseyness is, in a certain sense, simultaneous $n$-subtlety for all $n<\omega$, one could hope for Baumgartner's argument to somehow be adaptable to the context of our Question \ref{question:ramseylikeineffable}, and thus answer it positively. However, our attempts to do so have as yet been unsuccessful.

Let us close by posing the simplest version of Question \ref{question_Baumgartner_for_Ramsey_hierarchy} which remains open.
\begin{question}\label{question_simple}
Is the hypothesis ``$\exists\kappa\ \kappa\in\R^2([\kappa]^{<\kappa})^+$'' stronger in consistency strength than ``$\exists\kappa\ \kappa\in\R(\Pi^1_1(\kappa))^+$''?

\end{question}


\begin{thebibliography}{AHKZ77}

\bibitem[AHKZ77]{MR0460120}
F.~G. Abramson, L.~A. Harrington, E.~M. Kleinberg, and W.~S. Zwicker.
\newblock Flipping properties: a unifying thread in the theory of large
  cardinals.
\newblock {\em Ann. Math. Logic}, 12(1):25--58, 1977.

\bibitem[Bag19]{MR3894041}
Joan Bagaria.
\newblock Derived topologies on ordinals and stationary reflection.
\newblock {\em Trans. Amer. Math. Soc.}, 371(3):1981--2002, 2019.

\bibitem[Bau75]{MR0384553}
J.~E. Baumgartner.
\newblock Ineffability properties of cardinals. {I}.
\newblock pages 109--130. Colloq. Math. Soc. J\'anos Bolyai, Vol. 10, 1975.

\bibitem[Bau77]{MR0540770}
James~E. Baumgartner.
\newblock Ineffability properties of cardinals. {II}.
\newblock In {\em Logic, foundations of mathematics and computability theory
  ({P}roc. {F}ifth {I}nternat. {C}ongr. {L}ogic, {M}ethodology and {P}hilos. of
  {S}ci., {U}niv. {W}estern {O}ntario, {L}ondon, {O}nt., 1975), {P}art {I}},
  pages 87--106. Univ. Western Ontario Ser. Philos. Sci., Vol. 9. Reidel,
  Dordrecht, 1977.

\bibitem[BW]{brickhill-welch}
Hazel Brickhill and Philip Welch.
\newblock Generalisations of stationarity, closed and unboundedness and of Jensen's $\Box$.
\newblock (\emph{preprint}).

\bibitem[Cod]{CodyHigherIndescribability}
Brent Cody.
\newblock Higher indescribability and derived topologies.
\newblock (\emph{Under review}).

\bibitem[Cod20]{MR4206111}
Brent Cody.
\newblock A refinement of the {R}amsey hierarchy via indescribability.
\newblock {\em J. Symb. Log.}, 85(2):773--808, 2020.

\bibitem[Fen90]{MR1077260}
Qi~Feng.
\newblock A hierarchy of {R}amsey cardinals.
\newblock {\em Ann. Pure Appl. Logic}, 49(3):257--277, 1990.

\bibitem[For10]{MR2768692}
Matthew Foreman.
\newblock Ideals and generic elementary embeddings.
\newblock In {\em Handbook of set theory. {V}ols. 1, 2, 3}, pages 885--1147.
  Springer, Dordrecht, 2010.

\bibitem[Git07]{MR2710923}
Victoria Gitman.
\newblock {\em Applications of the proper forcing axiom to models of {P}eano
  arithmetic}.
\newblock ProQuest LLC, Ann Arbor, MI, 2007.
\newblock Thesis (Ph.D.)--City University of New York.

\bibitem[Git11]{MR2830415}
Victoria Gitman.
\newblock Ramsey-like cardinals.
\newblock {\em J. Symbolic Logic}, 76(2):519--540, 2011.

\bibitem[HL21]{MR4156888}
Peter Holy and Philipp L\"{u}cke.
\newblock Small models, large cardinals, and induced ideals.
\newblock {\em Ann. Pure Appl. Logic}, 172(2):102889, 50, 2021.

\bibitem[Hol]{HolyLCOandEE}
Peter Holy.
\newblock Ramsey-like operators.
\newblock Accepted at {\em Fund. Math.}, 2022.

\bibitem[Joh87]{MR918427}
C.~A. Johnson.
\newblock More on distributive ideals.
\newblock {\em Fund. Math.}, 128(2):113--130, 1987.

\bibitem[SW11]{MR2817562}
I.~Sharpe and P.~D. Welch.
\newblock Greatly {E}rd\H{o}s cardinals with some generalizations to the
  {C}hang and {R}amsey properties.
\newblock {\em Ann. Pure Appl. Logic}, 162(11):863--902, 2011.

\end{thebibliography}

\end{document}